\documentclass[reqno]{amsart}

\usepackage[usenames,dvipsnames]{xcolor}
\usepackage[colorlinks=true,urlcolor=blue,citecolor=blue,linkcolor=blue]{hyperref} 
\usepackage{graphics,graphicx}

\newcommand{\p}[1]{\ensuremath{\mathbf{P}\left(#1\right)}}
\newcommand{\e}[1]{\ensuremath{\mathbf{E}\left(#1\right)}}

\newtheorem{thm}{Theorem}[section]
\newtheorem{lem}[thm]{Lemma}
\newtheorem{prop}[thm]{Proposition}
\newtheorem{cor}[thm]{Corollary}
\newtheorem{dfn}[thm]{Definition}

\newcommand{\refT}[1]{Theorem~\ref{#1}}
\newcommand{\refC}[1]{Corollary~\ref{#1}}
\newcommand{\refL}[1]{Lemma~\ref{#1}}

\newcommand{\refP}[1]{Proposition~\ref{#1}}

\newcommand{\refand}[2]{\ref{#1} and~\ref{#2}}

\newcommand\cD{\mathcal D}

\newcommand\cF{\mathcal F}
\newcommand\cG{\mathcal G}
\newcommand\cH{\mathcal H}
\newcommand\cI{\mathcal I}

\newcommand\cK{\mathcal K}
\newcommand\cL{{\mathcal L}}
\newcommand\cM{\mathcal M}

\newcommand\cQ{\mathcal Q}
\newcommand\cR{{\mathcal R}}

\newcommand{\E}[1]{{\mathbf E}\left[#1\right]}

\newcommand{\I}[1]{{\mathbf 1}_{[#1]}}

\newcommand{\eqdist}{\ensuremath{\stackrel{\mathrm{d}}{=}}}

\newcommand{\hs}{\tau}

\newcommand{\hl}{\tau_\Iv}
\newcommand{\N}{\mathbb{N}}
\newcommand{\R}{\mathbb{R}}
\newcommand{\eps}{\varepsilon}
\newcommand{\Iv}{\cI}

\newcommand{\Zv}{\mathbf{0}}

\newcommand{\rv}{\mathbf{r}}
\newcommand{\xv}{\mathbf{x}}
\newcommand{\yv}{\mathbf{y}}
\newcommand{\rank}{\ensuremath{\mathrm{rank}}}

\begin{document}

\title{Hitting time theorems for random matrices}

\author{Louigi Addario-Berry}
\address{Department of Mathematics and Statistics, McGill University}
\email{louigi@math.mcgill.ca}
%\thanks{Louigi Addario-Berry was supported by an NSERC Discovery Grant for the duration of this research.}

\author{Laura Eslava}
\address{Department of Mathematics and Statistics, McGill University}
\email{laura.eslavafernandez@mail.mcgill.ca}

\date{April 5, 2013}

\maketitle
\begin{abstract}
Starting from an $n$-by-$n$ matrix of zeros, choose uniformly random zero entries and change them to ones, one-at-a-time, until the matrix becomes invertible. We show that with probability tending to one as $n \to \infty$, this occurs at the very moment the last zero row or zero column disappears. We prove a related result for random symmetric Bernoulli matrices, and give quantitative bounds for some related problems. These results extend earlier work by Costello and Vu. \cite{costello08rank}. 
\end{abstract}

\section{Introduction}

In this paper, we initiate an investigation of hitting time theorems for random matrix processes. Hitting time theorems have their origins in the 
study of random graphs; we briefly review this history, then proceed to an overview of recent work on discrete and continuous random matrix models and a statement of our results. 

To begin, consider the classical {\em Erd\H{o}s-R\'enyi graph process} $\{G_{n,p}\}_{0 \le p \le 1}$, defined as follows. Independently for each pair $\{i,j\} \subset [n]=\{1,\ldots,n\}$, let $U_{ij}$ be a Uniform$[0,1]$ random variable. Then, for $p \in [0,1]$ let $G_{n,p}$ have vertex set $[n]$ and edge set $\{\{i,j\}: U_{ij} \le p\}$. In $G_{n,p}$, each edge is independently present with probability $p$, and for $p < p'$ we have that $G_{n,p}$ is a subgraph of $G_{n,p'}$. 

Bollob\'as and Frieze \cite{MR860585} proved the following {\em hitting time theorem} for $G_{n,p}$, which is closely related to the main result of the present work. Let $\tau_{\delta \ge 1}$ be the first time $p$ that $G_{n,p}$ has minimum degree one, and let $\tau_{\textsc{pm}}$ be the first time $p$ that $G_{n,p}$ contains a perfect matching (or let $p=1$ if $n$ is odd). Then as $n \to \infty$ along even numbers, we have 
\[
\p{\tau_{\delta \ge 1}=\tau_{\textsc{pm}}} \to 1. 
\]
In other words, the first moment that the trivial obstacle to perfect matchings (isolated vertices) disappears, with high probability a perfect matching appears. Ajtai, Koml\'os and Szemer\'edi \cite{MR821516} had slightly earlier shown a hitting time theorem for Hamiltonicity; the first time $G_{n,p}$ has minimum degree two, with high probability $G_{n,p}$ is Hamiltonian. In fact, \cite{MR860585} generalizes this, showing that if $\tau_{\delta \ge 2k}$ is the first time $G_{n,p}$ has minimum degree $2k$ and $\tau_{k-\mathrm{Ham}}$ is the first time $G_{n,p}$ contains $k$ disjoint Hamilton cycles, then with high probability $\tau_{\delta \ge 2k} = \tau_{k-\mathrm{Ham}}$. Hitting time theorems have since been proved for a wide variety of other models and properties, including: connectivity \cite{MR860586}, $k$-edge-connectivity \cite{2275320,MR2116582} and $k$-vertex-connectivity \cite{2275320} in random graphs and in maker-breaker games;  connectivity in geometric graphs \cite{MR1442317}; and Hamiltonicity in geometric graphs and in the $d$-dimensional Gilbert model~\cite{MR2830612}. 

In this work we introduce the study of hitting time theorems for random discrete matrices. The study of random matrices is burgeoning, with major advances in our understanding over the last three to five years. Thanks to work by a host of researchers, the behaviour of the determinant \cite{MR2181645}, a wide range of spectral properties \cite{MR2859190,vu12wigner}, invertibility \cite{MR2557947}, condition numbers \cite{DBLP:conf/stoc/VuT07,MR2480613}, and singular values \cite{MR2827856,MR2265341} are now well (though not perfectly) understood. (This list of references is representative, rather than exhaustive.) The recent paper \cite{MR2752817} provides a nice collection of open problems, with a focus on random discrete matrices. 

In order to have a concept of hitting time theorems for matrices, we need to consider matrix {\em processes}, and we focus on two such processes. The first is the {\em Bernoulli} process $\cR_n=\{R_{n,p}\}_{0 \le p \le 1}$, defined as follows. 
Independently for each ordered pair $\{ij: 1 \le i \ne j \le n\}$, let $U_{ij}$ be a Uniform$[0,1]$ random variable. Then let $R_{n,p}$ be an $n$-by-$n$ matrix with $i,j$ entry $R_{n,p}(i,j)$ equal to one if $U_{ij} \le p$ and zero otherwise. For $n \ge 1$ and $0\le p \le 1$ we let $H_{n,p}$ be the directed graph with adjacency matrix $R_{n,p}$, so $\{H_{n,p}\}_{0 \le p \le 1}$ is a {\em directed Erd\H{o}s--R\'enyi graph process}. (We take $R_{n,p}$ to have zero diagonal entries as it is technically convenient for $H_{n,p}$ to have no loop edges; however, all our results for this model would still hold if the diagonal entries were generated by independent uniform random variables $\{U_{ii}:1 \le i \le n\}$, and with essentially identical proofs.) 

The second model we consider is the {\em symmetric} Bernoulli process $\cQ_n=\{Q_{n,p}\}_{0 \le p \le 1}$: with $U_{ij}$ as above, for $1 \le i < j \le n$ let $Q_{n,p}(i,j)=Q_{n,p}(j,i)=\I{U_{ij} \le p}$, and set all diagonal entries equal to zero. Throughout the paper, we work in a space in which $Q_{n,p}$ is the adjacency matrix of $G_{n,p}$ for each $0 \le p \le 1$. The principal result of this paper is to prove hitting time theorems for invertibility (or full rank) for both the Bernoulli matrix process and the symmetric Bernoulli process; we now proceed to state our new contributions in detail.

\section{Statement of results}\label{sec:statements}
Given a real-valued matrix $M$, write 
\[
Z^{\textsc{row}}(M) = \{i: \mbox{ all entries in row $i$ of $M$ equal zero}\},
\]
define $Z^{\textsc{col}}(M)$ similarly, and let $z(M)=\max(|Z^{\textsc{row}}(M)|,|Z^{\textsc{col}}(M)|)$. Given a collection of matrices $\cM=\{M_p\}_{0 \le p \le 1}$, let 
$\hs(\cM)=\inf\{p: z(M_p)=0\}$, with the convention that $\inf\emptyset=1$. We write $\hs=\hs(\{M_p\}_{0 \le p \le 1})$ when the matrix process under consideration is clear. 
We say that a square matrix $M$ is singular if $\det M=0$, and otherwise say that $M$ is non-singular. 
Our main result is the following. 
\begin{thm} \label{thm:main}As $n \to \infty$ we have 
\begin{align*}
\p{R_{n,\hs(\cR_n)}~\mbox{is non-singular}} & \to 1 \\
\p{Q_{n,\hs(\cQ_n)}~\mbox{is non-singular}} & \to 1. 
\end{align*}
\end{thm}
In proving Theorem~\ref{thm:main}, we also obtain the following new result, 
which states that for a wide range of probabilities $p$, with high probability 
there are no non-trivial linear dependencies in the random matrix $R_{n,p}$.
\begin{thm}\label{thm:digraph}
For any fixed $c > 1/2$, uniformly over $p \in (c\ln n/n,1/2)$, 
we have 
\[
\p{\mathrm{rank}(R_{n,p})=n-z(R_{n,p})} = 1-O((\ln\ln n)^{-1/2})\, .
\]
\end{thm}
The analogue of Theorem~\ref{thm:digraph} for the symmetric process $\cQ_n$ was established by Costello and Vu \cite{costello08rank}, and our analysis builds on theirs as  well as that of \cite{costello08symmetric}. The requirement that $c > 1/2$ in Theorem~\ref{thm:digraph} is necessary, since for $p=c\ln n/n$ with $ c < 1/2$, with probability $1-o(1)$ the matrix $R_{n,p}$ will contain two identical rows each with a single non-zero entry, as well as two identical columns each with a single non-zero entry; in this case $\rank(R_{n,p}) < n-z(R_{n,p})$. 

Our analyses of the processes $\cR_n$ and $\cQ_n$ are similar, but each presents its own difficulties. In the former, lack of symmetry yields a larger number of potential ``problematic configurations'' to control; in the latter, symmetry reduces independence between the matrix entries. Where possible, we treat the two processes in a unified manner, but on occasion different proofs are required for the two models. 
                         
There are two main challenges in proving Theorem~\ref{thm:main}. 
First, there are existing bounds of the form of Theorem~\ref{thm:digraph} for the symmetric Bernoulli process \cite{costello08rank}. 
However, in both models, $\tau$ is of order $\ln n/n + \Theta(1/n)$. This is rather diffuse; it means that the moment when the last zero row/column disappears is spread over a region in which $\Theta(n)$ ones appear in the matrix ($\Theta(n)$ new edges appear in the associated graph). As such, a straightforward argument from Theorem~\ref{thm:digraph} or its symmetric analogue, using a union bound,  is impossible. This is not purely a weakness of our methods. Indeed, if the matrix contains two identical {\em non-zero} rows then it is singular, and the probability there are two such rows (each containing a single non-zero entry, say) when $p\le \ln n/n$ is $\Omega(\ln^2 n/n)$. This is already too large for a naive union bound to succeed. Moreover, with current techniques there seems no hope of replacing our bound by one that is even, say, $\Omega(n^{\epsilon})$ for any positive $\epsilon$, so another type of argument is needed. 

The second challenge is that invertibility is not an increasing property (adding ones to a zero-one matrix can destroy invertibility). All existing proofs of hitting time theorems for graphs (of which we are aware) use monotonicity, usually in the following way. An {\em increasing graph property} is a collection $\cG$ of graphs, closed under graph isomorphism, and such that if $G \in \cG$ and $G$ is a subgraph of $H$, then $H \in \cG$. 
Suppose that $\cH$ and $\cK$ are increasing graph properties with $\cH \subset \cK$. 
If there is a function $f(n,p)$ such that uniformly in $0 < p < 1$, 
\[ 
\p{ G_{n,f(n,p)} \in \cH} = p+o(1) = \p{ G_{n,f(n,p)} \in \cK}\, ,
\]
then with probability $1-o(1)$, the first hitting times of $\cH$ and of $\cK$ coincide. This follows easily from the fact that $\cH$ and $\cK$ are increasing. However, it breaks down for non-increasing properties and there is no obvious replacement. 

To get around these two issues, we introduce a method for {\em decoupling} the event of having full rank from the precise time the last zero row or column disappears. This method is most easily explained in graph terminology. We take a subset of the vertices of the graph under consideration, and replace their (random) out- and/or in-neighbourhoods with deterministic sets of neighbours. We prove results about the modified model, and then show that by a suitable averaging out, we can recover results about the original, fully random model. We believe the results about the partially deterministic models are independently interesting, and we now state them. 

\begin{dfn}\label{dfn:template}
Given $n\ge 1$, a \emph{template} (or \emph{$n$-template}) is an ordered pair $\cL=(\cL^+,\cL^-)=((S_i^+)_{i \in I^+},(S_j^-)_{j \in I^-})$, where 
\begin{enumerate}
\item $I^+, I^-$ are subsets of $[n]$,
\item $(S_i^+)_{i \in I^+}$ and $(S_j^-)_{j \in I^-}$ are sequences of non-empty, pairwise disjoint subsets of $[n]$,
\item $\bigcup_{i\in I^+}S_i^+ \subset [n]\setminus I^-$ and  $\bigcup_{j\in I^-}S_j^-\subset [n]\setminus I^+$.
\end{enumerate} 
The \emph{size} of $\cL$ is 
$ \max(|I^+|,|I^-|,\max(|S_i^+|,i \in I^+),\max(|S_i^-|,i \in I^-))$. We write $\cI=\cI(\cL)=(I^+,I^-)$. 
Also, we say $\cL$ is {\em symmetric} if $\cL^+=\cL^-$. 
Finally, for $l \in \N$, we let $\cM^n(l)$ be the collection of $n$-templates of size at most $l$.
\end{dfn}
We remark there is a unique template $\cL$ of size zero, which satisfies $I^+=\emptyset=I^-$; we call this template {\em degenerate}.

Given $n\ge 1$, an undirected or directed graph $G$ on $n$ vertices and a template $\cL$ as defined above, let $G^\cL$ be the graph obtained from $G$ by letting each $i\in I^+$ have out-neighbours $S_i^+$ (and no others) and each $j\in I^-$ in-neighbours $S_j^-$ (and no others). Note that if $G$ is undirected and $\cL$ is symmetric, then $G^{\cL}$ is again undirected, provided we view a pair $uv,vu$ of directed edges as a single undirected edge. We write $Q_{n,p}^\cL$ and $R_{n,p}^\cL$ for the adjacency matrix of $G_{n,p}^\cL$ and $H_{n,p}^\cL$. In $Q_{n,p}^\cL$ and $R_{n,p}^\cL$, for $i\in I^+$ (resp. $i \in I^-$), the non-zero entries of row $i$ are precisely those with indices in $S^+_i$ (resp. in $S^-_i$).

\begin{thm}\label{thm:gnpl}
Fix $K \in \N$ and $c > 1/2$. 
For any $p \in (c\ln n/n,1/2)$ and any template $\cL \in \cM^n(K)$, 
\[
\p{\mathrm{rank}(R_{n,p}^\cL)=n-z(R_{n,p}^\cL)} = 1-O((\ln\ln n)^{-1/2})\, . 
\]
If, additionally, $\cL$ is symmetric then 
\[
\p{\mathrm{rank}(Q_{n,p}^\cL)=n-z(Q_{n,p}^\cL)} = 1-O((\ln\ln n)^{-1/4})\, . 
\]
\end{thm}
We briefly remark on the assertions of the latter theorem. First, the first probability bound immediately implies Theorem~\ref{thm:digraph}, by taking $I^+$ and $I^-$ to be empty. 
Next, the condition of pairwise disjointness is necessary. To see this, note that if vertices 
$u,v$ have degree one and have a common neighbour $w$ then the rows of the adjacency 
matrix corresponding to $u$ and $w$ are identical, creating a non-degenerate linear relation. 
Finally, in proving the theorem we in fact only require that if $K=\max_{i \in I^+} |S_i^+|$, then $|I^+| \cdot K = o(p^{-2}/(n\ln n))$ (and similiarly for the maximum size of $S_i^-$). As we do not believe this condition 
is optimal we have opted for a more easily stated theorem. However, it would be interesting to know how far the boundedness condition could be weakened. 

The proof of Theorem~\ref{thm:gnpl} is based on an analysis of an iterative exposure of minors; the first use of such a procedure to study the rank of random matrices was due to Koml\'os \cite{MR0221962}. In brief, we first show that for suitably chosen $n' < n$, a fixed $n'$-by-$n'$ minor of $R_{n,p}$ is likely to have nearly full rank. Adding the missing rows and columns one-at-a-time, we then show that any remaining dependencies are likely to be ``resolved'', i.e. eliminated, by the added rows. Our argument is similar to that appearing in \cite{costello08rank} for the study of $Q_{n,p}$, but there are added complications due to the fact that our matrices are partially deterministic on the one hand, and asymmetric on the other. The proof of Theorem~\ref{thm:gnpl} occupies a substantial part of the paper; a somewhat more detailed sketch appears in Section~\ref{sec:iterative}. 

Vershynin \cite{RSA:RSA20429} has very recently strengthened the bounds of 
Costello, Tao and Vu \cite{costello08symmetric}, showing that for a broad range of dense 
symmetric random matrices, the singularity probability decays at least as quickly as 
$e^{-n^\beta}$, for some (model-dependent) $\beta>0$. It seems plausible (though not certain) that Vershynin's techniques could be transferred to the current, sparse setting, to yield bounds of the form $1-O(e^{-(\ln\ln n)^\beta})$ in Theorems~\ref{thm:digraph} and~\ref{thm:gnpl}. However, we believe, and the results of \cite{MR2607371} suggest, that aside from zero-rows, the most likely cause of singularity is two identical rows, each containing a single one. If the latter is correct then it should in fact be possible to obtain bounds of the form $1-O(n^{1-2c})$ (for $c > 1/2$ as above); as alluded to earlier, for the moment such bounds seem out of reach.

\section*{Notation}\label{sec:notation}
Before proceeding with details, we briefly pause to introduce some terminology. 
Given an $m\times m$ matrix $M=(m_{ij})_{1\le i,j\le m}$, 
the {\em deficiency} of $M$ is the quantity 
\[Y(M):=m-rank(M)-z(M).\]

Also, for any $i,j\in [m]$ we denote by $M^{(i,j)}$ the matrix obtained by removing the $i$-th row of $M$ and the $j$-th column; we refer to $M^{(i,j)}$ as the {\em $(i,j)$ minor} of $M$. More generally, given $A,B \subset [m]$ we write $M^{(A,B)}$ for the matrix obtained from $M$ by removing the rows indexed by $A$ and the columns indexed by $B$. Also, for $1\le k\le m$ we write $M[k]=(M_{ij})_{1\le i,j\le k}$.

For a graph $G=(V,E)$ and $v \in V$, we write $N_G^+(v)$ for the set of out-neighbours of $v$ in $G$ and $N_G^-(v)$ for the set of in-neighbours of $v$ in $G$. (If $G$ is undirected then we write $N^+_G(v)=N^-_G(v)=N_G(v)$ for the set of neighbours of $G$.)
If $M$ is the adjacency matrix of $G$ then for $1\le i\le m$ we write $N^+_M(i)=N^+_G(i)$, and similarly for $N^-_M(i)$ (and $N_M(i)$ if $M$ is symmetric). Note that in this case, $Z^{\textsc{row}}(M)$ and $Z^{\textsc{col}}(M)$ correspond to the sets $\{v\in V\, :\, |N_G^+(v)|=0\}$ and $\{v\in V\, :\, |N_G^-(v)|=0\}$, respectively.

Given real random variables $X$ and $Y$ we write $X \preceq_{\mathrm{st}} Y$ if for all $t \in \mathbb{R}$, $\p{X \ge t} \le \p{Y \ge t}$, and in this case say that $Y$ {\em stochastically dominates} $X$. Finally, we omit floors and ceilings for readability whenever possible. 

\section*{Outline}
The structure of the remainder of the paper is as follows. In Section~\ref{sec:decouple} we explain how to ``decouple'' the linear dependencies of the matrix process from the time at which the last zero row or zero column disappears. This decoupling allows us to prove Theorem~2.1 assuming that Theorem~2.4 holds. 

In Section~\ref{sec:iterative} we introduce the {\em iterative exposure of minors}, analogous to the vertex exposure martingale for random graphs, which we use to study how $Y(R_{n,p}[m])$ changes as $m$ increases. We then state a key ``coupling'' lemma (Lemma~\ref{lem:coupling}), which asserts that for some $n' < n$ with $n-n'$ sufficiently small, the process $(Y(R_{n,p}[m]),n' \le m \le n)$ is stochastically dominated by a reflected simple random walk with strongly negative drift. Postponing the proof of this lemma, we then show how Theorem~2.4 follows by standard simple hitting probability estimates for simple random walk. 

In Section~\ref{sec:rank}, we describe ``good'' structural properties, somewhat analogous to graph expansion, that we wish for the matrices $R_{n,p}[m]$ to possess. 
The properties we require are tailored to allow us to apply {\em linear and quadratic Littlewood-Offord bounds} on the concentration of random sums. 
Proposition~\ref{prop:key}, whose proof is postponed, states that these properties hold with high probability throughout the iterative exposure of minors. Assuming the properties hold, it is then a straightforward matter to complete the proof of Lemma~\ref{lem:coupling}. 

In Section~\ref{sec:structure} we complete the proof of Proposition~\ref{prop:key}. This, the most technical part of the proof, is most easily described in the language of random graphs rather than of random matrices. It is largely based on establishing suitable expansion and intersection properties of the neighbourhoods that hold for all ``small'' sets of vertices and in each of the graphs $R_{n,p}[m]$ considered in the iterative exposure of minors. 

Finally, Appendix A states standard binomial tail bounds that we use in our proofs, and Appendix B contains proofs of results that are either basic but technical, or that essentially follow from previous work but do not directly apply in our setting. 

\section{Decoupling connectivity from rank estimates: the proof of Theorem~2.1 from Theorem~2.4}\label{sec:decouple}
In this section we explain how the first assertion of \refT{thm:main} follows from the first assertion of \refT{thm:gnpl}. A similar argument applies to the second assertion of \refT{thm:main}; we comment on the necessary adjustments to the argument in Section~\ref{subsec:decoupling_sym}.

Recall that the process $\cR_n$ is generated by a family $\{U_{ij}:1 \le i\ne j \le n\}$ of independent Uniform$[0,1]$ random variables. Given $I^+,\, I^- \subset [n]$, let $\Iv=(I^+,I^-)$ and write
\[
\begin{split}
\cF_\Iv & =\sigma (\{U_{ij}:\, i\in I^+ \text{ or } j\in I^-\}), \\
\cG_\Iv & =\sigma (\{U_{ij}:\, i\in [n]\setminus I^+ \text{ and } j\in [n]\setminus I^-\}).
\end{split}
\]
Informally, $\cF_\Iv$ contains all information that can be determined from the process $\cR_n$ by observing only rows with indices in $I^+$ and columns with indices in $I^-$. All information about all remaining entries is contained in $\cG_\Iv$.

Next, given $p\in (0,1)$, let 
\[
A_\Iv(p) = \{z(R_{n,p}^{(I^+,I^-)})=0 \}\, ,
\]
and let 
\[
B_\Iv(p) = \{I^+ \subset Z^{\textsc{row}}(R_{n,p}),I^- \subset Z^{\textsc{col}}(R_{n,p})\}\, .
\]
In words, $A_\Iv(p)$ is the event that the matrix obtained from $R_{n,p}$ by deleting the rows indexed by $I^+$ and the columns indexed by $I^-$  has neither zero rows nor zero columns.
We remark that $A_\Iv(p)$ and $B_\Iv(p)$ are measurable with respect to $\cG_\Iv$ and $\cF_\Iv$, respectively, and that $A_\Iv(p)\cap B_\Iv(p)$ is precisely the event that $Z^{\textsc{row}}(R_{n,p})=I^+$ and $Z^{\textsc{col}}(R_{n,p})=I^-$. We write $A_\Iv,B_\Iv$ instead of $A_\Iv(p), B_\Iv(p)$ when the dependence on $p$ is  clear from context. 

Finally, let 
\begin{align*}
\hl=\hl(\cR_n) & =\min \{p\in (0,1): Z^{\textsc{row}}(R_{n,p})\cap I^+=\emptyset=Z^{\textsc{col}}(R_{n,p})\cap I^-\}\, .
\end{align*}
Then, for any template $\cL=((S_i^+)_{i\in I^+},(S_j^-)_{j\in I^-})$, let 
\[
C_\cL = \{\forall ~ i \in I^+, N_{R_{n,\hl}}^+(i)=S_i^+\} \cap \{ \forall ~ j \in I^-,N_{R_{n,\hl}}^-(j)=S_j^-\}\, .
\]
Observe that $C_\cL$ and $\hl$ are measurable with respect to $\cF_\Iv$. Furthermore, the entries that are random in $R_{n,p}^\cL$ are precisely those corresponding to the random variables generated by $\cG_\Iv$. Lemma~\ref{Split}, below, uses this fact 
in order to express the conditional distribution of $\hs$ given $A_\Iv,B_\Iv$, and $C_\cL$, as an integral against a conditional density function. 

\begin{lem}\label{Split}
Fix $1 \le l \le n$ and $p\in (0,1)$. Then for any non-degenerate $n$-template $\cL=((S_i^+)_{i \in I^+},(S_i^-)_{i \in I^-})$, letting $\Iv=(I^+,I^-)$ and writing $\tau=\tau(\cR_n)$, we have 
\begin{align*}
& \quad \p{Y(R_{n,\hs})=0~|~A_\Iv(p),B_\Iv(p),C_\cL}\\
=& \quad \int_{0}^{1} \p{Y(R_{n,t}^\cL)=0~|~A_\Iv(p)} f(t~|~B_\Iv(p),C_\cL)\mathrm{d}t,
\end{align*}
where $f(\cdot~|~B_\Iv(p),C_\cL)$ is the conditional density of $\hl$ given $B_\Iv(p)$ and $C_\cL$.
\end{lem}
In proving Lemma~\ref{Split} we use the following basic observation. Given independent $\sigma$-algebras $\cF_1$ and $\cF_2$, for every $E_1,F_1\in \cF_1$ and $E_2,F_2\in \cF_2$, we have
\[\p{E_1,E_2~|~F_1,F_2}=\p{E_1~|~F_1}\p{E_2~|~F_2}.\]
\begin{proof}
If $A_\Iv,B_\Iv$ and $C_\cL$ all occur, we necessarily have $rank(R_{n,\hs})=rank(R_{n,\hs}^\cL)$ and $\hl=\hs$. For any $K\in \N$, 
we may thus rewrite $\p{Y(R_{n,\hs})=0~|~A_\Iv,B_\Iv,C_\cL}$ as 
\begin{equation*}
\sum_{i=0}^{K-1} \p{Y(R_{n,\hl}^\cL)=0, \hl \in [\tfrac{i}{K},\tfrac{i+1}{K})~|~A_\Iv,B_\Iv,C_\cL}.
\end{equation*}
For any $0\le i< K$, if $\frac{i}{K}\leq \hl < \frac{i+1}{K}$ and no edges arrive in the interval $(\hl, \frac{i+1}{K}]$,
then $R_{n,\hl}^\cL$ and $R_{n,\frac{i+1}{K}}^\cL$ are identical.
Writing $D$ for the event that a pair of distinct edges arrive within one of the intervals $\{[\frac{i}{K},\frac{i+1}{K}],\, 0\le i<K\}$, 
it follows that
\begin{align}
& \bigg|\p{Y(R_{n,\hs})=0~|~A_\Iv,B_\Iv,C_\cL} \qquad\qquad\qquad\qquad\qquad \nonumber \\
& - \sum_{i=0}^{K-1} \p{Y(R_{n,\frac{i+1}{K}}^\cL), \hl\in [\tfrac{i}{K},\tfrac{i+1}{K})~|~A_\Iv,B_\Iv,C_\cL}\bigg|  \nonumber\\ 
\le & \quad \p{R_{n,\hl}^\cL\neq R_{n,\hl+\frac{1}{K}}^\cL~|~A_\Iv,B_\Iv,C_\cL} 
\le \frac{\p{D}}{\p{A_\Iv,B_\Iv,C_\cL}}\, . \label{eq:finesplit}
\end{align}
For fixed edges $e$ and $e'$, we have $\p{|U_e-U_{e'}|\leq \frac{1}{K}}\leq \frac{2}{K}$. 
By a union bound it follows that $\p{D}\leq \frac{2n(n-1)}{K}$, and 
so the final term in (\ref{eq:finesplit}) tends to $0$ as $K \to \infty$. 

Finally, by the observation about conditional independence just before the start of the proof, for all $K$ we have 
\begin{align*}
& \sum_{i=0}^{K-1}\p{Y(R_{n,\frac{i+1}{K}}^\cL), \hl\in [\tfrac{i}{K},\tfrac{i+1}{K})~|~A_\Iv,B_\Iv,C_\cL} \\
= &\sum_{i=0}^{K-1} \p{Y(R_{n,\frac{i+1}{K}}^\cL)~|~A_\Iv}\p{\hl \in [\tfrac{i}{K},\tfrac{i+1}{K})~|~B_\Iv, C_\cL}\, ,
\end{align*}
and taking $K \to \infty$ completes the proof. 
\end{proof}

The next definition captures the event that, for a given $p < \hs$, the rows (resp.\ columns) of $R_{n,\hs}$ indexed by $Z^{\textsc{row}}(R_{n,p})$ (resp.\ $Z^{\textsc{col}}(R_{n,p})$) are such that Theorem~\ref{thm:gnpl} can be applied. 
\begin{dfn}\label{dfn:D_k}
Given $0 < p < p' < 1$ and $K\in \N$, let $\cD_{K}(p,p')$ be the event that $((N_{R_{n,p'}}^+(i))_{i\in Z^{\textsc{row}}(R_{n,p})},(N_{R_{n,p'}}^-(i))_{i\in Z^{\textsc{col}}(R_{n,p})})$ is a non-degenerate $n$-template of size at most $K$.
\end{dfn}

\begin{lem}\label{lem:rest}
For any $\eps>0$ there exists $a>0$ and integer $K=K(a)>0$ such that, setting $p_1=\frac{\ln n-a}{n}$ and $p_2=\frac{\ln n+a}{n}$, for all $n$ sufficiently large, 
\begin{align*}
\p{\hs(\cR_n) > p_2} &\le \eps \, ,\\
\p{\cD_{K}(p_1,\hs(\cR_n))} &\geq 1-\eps\,.
\end{align*}
\end{lem}
The proof of \refL{lem:rest} is straightforward but technical, and is presented in Appendix~\ref{Sec:easy_proofs}.

Now recall the definition of $Y(M)$ from the notation section on page \pageref{sec:notation}. Observe that the first claim of Theorem~\ref{thm:main} is equivalent to the statement that with high probability $Y(R_{n,\hs(\cR_n)})=0$. 
Therefore, to establish the first assertion of \refT{thm:main}, it suffices to prove the following theorem. 

\begin{thm}\label{thm:ygnp}
$
\p{Y(R_{n,\tau(\cR_n)})=0} \to 1 \text{ as }n \to \infty.
$
\end{thm}
\begin{proof}
Fix $\eps > 0$, let $a>0$ and $K=K(a)>0$ be as in \refL{lem:rest}. Throughout the proof write $p_1=\frac{\ln n - a}{n}$, $p_2=\frac{\ln n + a}{n}$, and $\tau=\tau(\cR_n)$. 
Note that $\cD_K(p_1,\hs)$ occurs precisely if there exists a non-degenerate
$\cL\in \cM^n(K)$ such that $A_{\Iv(\cL)}(p_1),B_{\Iv(\cL)}(p_1)$, and $C_\cL$ all occur. Furthermore, if $\cL\ne \cL'$ then $A_{\Iv(\cL)}(p_1)\cap B_{\Iv(\cL)}(p_1)\cap C_\cL$ and $A_{\Iv(\cL')}(p_1)\cap B_{\Iv(\cL')}(p_1)\cap C_{\cL'}$ are disjoint events. 
Writing $\hat{\cM}^n(K) = \{\cL \in \cM^n(K): \cL~\mbox{is non-degenerate}\}$,  it follows that
\begin{align*}
	&\quad \p{Y(R_{n,\hs})=0} \\
\ge 	& \quad \p{Y(R_{n,\hs})=0,\cD_K(p_1,\hs)} \\
= 	& \quad \sum_{\cL \in \hat{\cM}^n(K)} \p{Y(R_{n,\hs})=0,A_\Iv,B_\Iv,C_\cL}
\end{align*}
We will show that for any $\cL \in \hat{\cM}^n(K)$, 
\begin{equation}\label{eq:toprove_frac}
\frac{\p{Y(R_{n,\hs})=0,A_\Iv,B_\Iv,C_\cL}}{\p{A_\Iv,B_\Iv,C_\cL,\hs \le p_2}} \ge 1-o(1)\, .
\end{equation}
Assuming this, it follows that 
\begin{align*}
	&\quad \p{Y(R_{n,\hs})=0}\\
 \ge 	&\quad (1-o(1)) \sum_{\cL \in \hat{\cM}^n(K)} \p{A_\Iv,B_\Iv,C_\cL,\hs \le p_2} \\
 =	& \quad (1-o(1)) \p{\cD_K(p_1,\hs),\hs \le p_2} \\
 \ge	& \quad 1-2\eps \, 
\end{align*}
for $n$ large, the last inequality by Lemma~\ref{lem:rest}. Since $\eps > 0$ was arbitrary, it thus remains to prove (\ref{eq:toprove_frac}), for which we use 
Lemma~\ref{Split}.

Fix $\cL \in \hat{\cM}^n(K)$ and let 
$N=n-|I^+|\geq n-K$. Then $p_1=\frac{\ln n-a}{n}=\frac{\ln N-a}{N}+O(\frac{1}{N^2})$. 
For any fixed integer $i \ge 1$ and distinct $v_1,\ldots,v_i \in [n]\setminus (I^+\cup \bigcup_{i \in I^-} S_i^-)$, 
\[
\p{\bigcap_{j=1}^i \{N^+_{R_{n,p_1}^{\cL}}(v_j)=\emptyset\}} 
= (1-p_1)^{i(N-i)+i(i-1)}
=(1+o(1)) (1-p_1)^{N\cdot i}\, .
\]
Furthermore, for $j \in I^+ \cup \bigcup_{i \in I^-} S_i^-$ necessarily $N^+_{R_{n,p_1}^{\mathcal{L}}}(v_j) \ne \emptyset$ since $\cL$ is non-degenerate. It follows by the method of moments (see \cite{MR1782847}, Chapter 6) that 
$|Z^{\textsc{row}}(R_{n,p}^{\cL})|$ is asymptotically Poisson($e^a$). The same argument establishes that $|Z^{\textsc{col}}(R_{n,p}^{\cL})|$ has the same asymptotic distribution. It follows that $\p{A_\Iv}\le 2e^{-e^a} +o(1)$, so by the first assertion of \refT{thm:gnpl}, for $p\in (p_1,p_2)$ we have 
\begin{equation} \label{withAl}
\p{Y(R_{n,p}^\cL)>0~|~A_\Iv}\le 
\frac{\p{Y(R_{n,p}^\cL)>0}}{\p{A_\Iv}}=\frac{O(e^{e^a})}{(\ln\ln n)^{1/2}}.
\end{equation}

Since $a=O(1)$, by Lemma~\ref{Split} we thus have
\begin{align*}
	& \p{Y(R_{n,\hs})=0~|~A_\Iv,B_\Iv,C_\cL} \\
\geq	& \int_{p_1}^{p_2} \p{Y(R_{n,t}^\cL)=0~|~A_\Iv} f(t~|~B_\Iv,C_\cL)\mathrm{d}t,, \\ 
\geq	& \left(1-o(1)\right) \int_{p_1}^{p_2} f(t~|~B_\Iv, C_\cL)\mathrm{d}t, \\
=	&\left(1-o(1)\right) \p{\hl\in[p_1,p_2]~|~B_\Iv, C_\cL}. 
\end{align*}
Multiply both sides of the preceding inequality by $\p{A_\Iv,B_\Iv,C_\cL}$. Since $A_\Iv$ is 
independent from $B_\Iv$ and $C_\cL$, we obtain 
\begin{align*}
	& \quad \p{Y(R_{n,\hs})=0,A_\Iv,B_\Iv,C_\cL} \\
\ge 	& \quad (1-o(1)) \p{A_\Iv,B_\Iv,C_\cL,\hl \in [p_1,p_2]}
\end{align*}
Finally, if $A_\Iv,B_\Iv$, and $C_\cL$ all occur then necessarily $\hl=\hs$ and $\hs > p_1$, so 
we may replace $\{\hl \in [p_1,p_2]\}$ by $\{\hs \le p_2\}$ in the final probability, 
and (\ref{eq:toprove_frac}) follows. The proof is complete.
\end{proof}

\subsection{Notes on the proof of \refT{thm:main} for $G_{n,p}$} \label{subsec:decoupling_sym}

The decoupling of connectivity from the rank estimates of $R_{n,p}$ is not extremely sensitive to the structure of $R_{n,p}$ except through Theorem~\ref{thm:gnpl}, 
and the broad strokes of the argument of this section are therefore unchanged. In particular, define $\cF_{\Iv}$ and $\cG_{\Iv}$ as before (but recall that 
only the variables $\{U_{ij}\}_{1\le i<j\le n}$ are independent). Then Lemma~\ref{Split} holds under the additional restriction that the template $\cL$ is {\em symmetric}, as in this case $I^+=I^-$ and the $\sigma$-algebras $\cF_{\Iv}$ and $\cG_{\Iv}$ are indeed symmetric. We replace the event $D_K(p,p')$ with the event that 
$((N_{Q_{n,p'}}(i))_{i \in Z(Q_{n,p})},(N_{Q_{n,p'}}(i))_{i \in Z(Q_{n,p})})$ is a non-degenerate $n$-template of size at most $K$ (in which case it is necessarily symmetric). Lemma~\ref{lem:rest} then holds with $\hs(\cR_n)$ replaced by $\hs(\cQ_n)$, with an essentially identical proof.
Assuming the second bound in Theorem~\ref{thm:gnpl}, the rest of the proof then follows without substantial changes.

\section{Analysis of the iterative exposure process: the proof of \refT{thm:gnpl} 
modulo a coupling lemma}\label{sec:iterative}

To prove Theorem~\ref{thm:gnpl}, we analyze an iterative exposure of minors of the matrix. 
(In other words, we will expose the edges incident to the vertices of $H_{n,p}^\cL$ in a vertex-by-vertex fashion.)
This strategy was first used in the context of random symmetric matrices in \cite{costello08symmetric}, to show 
that random symmetric Bernoulli$(1/2)$ matrices are almost surely non-singular. 

For the remainder of the paper, $c \in (1/2,1)$ is a fixed constant, 
and $c\ln n/n \le p \le 1/2$. 
Also, for the rest of the paper, let $\alpha<1$ be such that $\alpha c \in (1/2,3/4)$, write $\gamma=\alpha c-1/2$,
and let $n' =\lceil \alpha n \rceil$. Given any integer $K \ge 1$, for $n$ sufficiently large and any $n$-template $\cL=((S_i^+)_{i \in I^+},(S_i^-)_{i \in I^-}) \in \cM^n(K)$ 
by permuting the rows and columns of $R_{n,p}^{\cL}$ we may assume that 
\[
I^+ \cup I^- \cup \bigcup_{i \in I^+} S_i^+ \cup \bigcup_{i \in I^-} S_i^- \subset [n']\,  ; 
\]
we call such $\cL\in \cM^n(K)$ \emph{permissible}. We work only with permissible templates to ensure that 
in the iterative exposure of minors in $R_{n,p}^{\cL}$ starting from $R_{n,p}^{\cL}[n']$, all new off-diagonal matrix entries whose row (resp.\ column) is not in $\bigcup_{i \in I^+} S_i^+$ (resp.\ $\bigcup_{i \in I^-} S_i^-$) are Bernoulli$(p)$ distributed. 
We begin by showing that $R_{n,p}^\cL[n']$ is extremely likely to have quite large rank. 

\begin{lem}\label{almost fullrank}
For any $\eps > 0$ and $c\ln n/n \le p \le 1/2$, there exists a constant $c_1$ such that 
\begin{align*}
\p{\rank(R_{n,p})\geq (1-\eps) n}\geq 1-O(n^{-c_1n}).
\end{align*}
\end{lem}
The proof of an analogous bound for random symmetric sparse Bernoulli matrices appears in \cite{costello08rank}. 
\begin{proof}
Denote the rows of $R_{n,p}$ by $\rv_1,\ldots, \rv_n$. Let $S= \mathrm{span}(\rv_1,\ldots,\rv_{\lfloor (1-\eps) n\rfloor})$ and $d=\dim S$.
Let $R$ be the event that 
$\rv_i\in S$ for every $\lfloor (1-\eps) n\rfloor<i\le n$. 
By symmetry we have
\begin{equation}\label{eq:symmetry}
\p{\rank(R_{n,p})\le (1-\eps)n}\le {n \choose \eps n }\p{R},
\end{equation}
so we now focus on bounding $\p{R}$. By relabelling (i.e.,\ replacing $R_{n,p}$ by $UR_{n,p}U^{-1}$ for some permutation matrix $U$), we may assume $R_{n,p}$ is a block matrix
\[R_{n,p}=\left( \begin{array}{cc}
A & B \\
C & D \end{array} \right),\]
where $A$ is an $\lfloor (1-\eps) n\rfloor \times d$ matrix with $\rank(A)=d \le (1-\eps) n$. Thus, the columns in $B$ are in the span of the columns in $A$, so $AG=B$ for some (unique) matrix $G$. 

On the other hand, $D$ is an $\lceil \eps n\rceil \times (n-d)$ matrix and $d\le \eps n$. It follows that there exists $c_2>0$ such that for any fixed matrix $M$,
\[\p{D=M}\le (1-p)^{(\eps n)^2-\eps n}\le e^{-\eps^2c n\ln n+\eps n/2}\le n^{-c_2n},\]
The first inequality holds since $D$ has at least $(\eps n)^2-\eps n$ independent Bernoulli($p$) entries.

Now, if $R$ holds, then there exists a matrix $F$ such that both $FA=C$ and $FB=D$ hold. 
Furthermore, note that if $F'$ also satisfies $F'A=C$ then 
\begin{equation*}
FB=FAG=F'AG=F'B,
\end{equation*}
so if $R$ occurs then $D=FB$ is uniquely determined by $A, B$ and $C$. 
Consequently, for any such $F$ we have 
\[
\p{R~|~ A,B,C} \le \p{D=FB~|~A,B,C} \le n^{-c_2 n}\, .
\]
Since $\p{R}=\E{\p{R~|~A,B,C}}$, by \eqref{eq:symmetry} it follows that 
\[\p{\rank(R_{n,p})\le (1-\eps)n}\le {n \choose \eps n }\p{R}\le \left(\frac{e}{\eps}\right)^{\eps n} n^{-c_2 n}.\]
\end{proof}

In fact, Lemma~\ref{almost fullrank} also holds for $R_{n,p}^{\cL}[n']$ as well. To see this, observe that since $\cL$ has size at most $K$, $|\rank(R_{n,p}^{\cL}[n'])-\rank(R_{n,p}[n'])|\le K^2=O(1)$, and $R_{n,p}[n']$ has the same distribution as $R_{n',p}$. We thus obtain the following corollary.

\begin{cor}\label{rank_firstminor}
For any $K \in \N$ and $\eps > 0$, there exists a constant $c_1$ such that uniformly over $\cL \in \cM^n(K)$ 
and $c\ln n/n \le p \le 1/2$, 
\begin{align*}
\p{\rank(R_{n,p}^\cL[n'])\geq (1-\eps) n}\geq 1-O(n^{-c_1n}).
\end{align*}
\end{cor}

We next consider how the deficiency $Y(R_{n,p}^\cL[m])$ drops as $m$ increases from $n'$ to $n$. We have 
\begin{align}
Y(R_{n,p}^\cL[m+1]) & = Y(R_{n,p}^\cL[m])+1\nonumber \\
& +(z(R_{n,p}^\cL[m])-z(R_{n,p}^\cL[m+1])) \nonumber \\
&  -  (\rank(R_{n,p}^\cL[m+1])-\rank(R_{n,p}^\cL[m])),\label{eq:ychange}
\end{align}
so $Y$ decreases as the rank increases and, on the other hand, increases when zero rows or zero columns disappear causing that $z(R_{n,p}^\cL[m])>z(R_{n,p}^\cL[m+1])$. 
To show that $Y(R_{n,p}^\cL)$ is likely zero, we will couple $(Y(R_{n,p}^\cL[m]),n' \le m \le n)$ to a simple random walk 
with strongly negative drift, in such a way that with high probability the random walk provides an upper bound for 
$(Y(R_{n,p}^\cL[m]),n' \le m \le n)$. Of course, showing that such a coupling exists involves control on the rank increase and on the decrease in $z(R_{n,p}^\cL[m])$ as $m$ increases from $n'$ to $n$. 
Observe that we always have $\rank(R_{n,p}^\cL[m]) \le \rank(R_{n,p}^\cL[m+1])\leq \rank(R_{n,p}^\cL[m])+2$ 
since $R_{n,p}^\cL[m]$ may be obtained from $R_{n,p}^\cL[m+1]$ by deleting a single row and column. 
It follows from \eqref{eq:ychange} that if $z(R_{n,p}^\cL[m])=z(R_{n,p}^\cL[m+1])$ 
then $Y(R_{n,p}^\cL[m+1])-Y(R_{n,p}^\cL[m]) \in \{-1,0,1\}$. Also, if $z(R_{n,p}^\cL[m])=z(R_{n,p}^\cL[m+1])-1$ then necessarily $\rank(R_{n,p}^\cL[m+1])\ge \rank(R_{n,p}^\cL[m])+1$ and so $Y(R_{n,p}^\cL[m+1])-Y(R_{n,p}^\cL[m]) \in \{0,1\}$. Together, this shows that $|Y(R_{n,p}^\cL[m+1])-Y(R_{n,p}^\cL[m])|\le 1$ whenever $z(R_{n,p}^\cL[m])-z(R_{n,p}^\cL[m+1]) \le 1$,

Establishing further control on the rank increase is rather involved, and is the primary work of Sections~\refand{sec:rank}{sec:structure}. 
It will turn out that typically, $Y(R_{n,p}^\cL[m+1])-Y(R_{n,p}^\cL[m]) = -1$ when $Y(R_{n,p}^\cL[m]) > 0$, and $Y(R_{n,p}^\cL[m+1])=0$ when $Y(R_{n,p}^\cL[m])=0$. More precisely, we have the following lemma.
\begin{lem}\label{lem:coupling} 
For fixed $K \in \N$, there exists $C > 0$ such that the following holds. 
Given integer $n \ge 10$, let $\beta=\beta(n)=C(\ln\ln n)^{-1/2}$. 
Then uniformly over $\cL \in \cM^n(K)$ and $c\ln n/n \le p \le 1/2$, 
there exists a coupling of $(Y(R_{n,p}^\cL[m]),n' \le m \le n)$ and a collection 
$(X_m,n'\le m< n)$ of iid random variables with $\p{X_i=1}=\beta$ and $\p{X_i=-1}=1-\beta$, 
such that with probability $1-O(n^{-\gamma/2})$, for all $n'\le m< n$, 
\[
Y(R_{n,p}^\cL[m+1])- Y(R_{n,p}^\cL[m]) \le \begin{cases}
					X_{m}	& \mbox{ if } Y(R_{n,p}^\cL[m])>0 \\
					\max(X_{m},0)	& \mbox{ if } Y(R_{n,p}^\cL[m])=0\, .
				\end{cases}
\]
\end{lem}
The proof of Lemma~\ref{lem:coupling} occupies much of the remainder of the paper. 
We say the coupling in the preceding lemma {\em succeeds} if for all $n'< m \le n $, the final inequality holds. 

Now fix $(X_i, i \ge 1)$ iid random variables with $\p{X_i=1}=\beta$ and $\p{X_i=-1}=1-\beta$. Set $S_0=0$, and for $k \ge 1$ let $S_k = \sum_{i=1}^k X_i$. 
We call $(S_k,k \ge 0)$ a {\em $\beta$-biased simple random walk} (SRW). 
Also, for $k \ge 1$ let $M_k= \min_{0 \le i \le k} S_i$, and let $D_k=S_k-M_k$. 

Observe that when $S_k$ is not at a new global minimum, $D_{k+1}$ is either $D_k+1$ (with probability $\beta$) or $D_k-1$. On the other hand, when 
$S_k$ is at a new minimum then $D_k=0$, and either $D_{k+1} = 1$ (again with probability $\beta$) or $D_{k+1}=0$. 

Now imagine for a moment that $Y(R_{n,p}^\cL[n'])=0$. In this case, in view of the preceding paragraph, if the coupling succeeds then we have $D_k \ge Y(R_{n,p}^\cL[n'+k])$ for all $0 \le k \le n-n'$. It follows that if $Y(R_{n,p}^\cL[n'])$ happens to equal zero then we can bound $\p{Y(R_{n,p}^\cL[n]) > 0}$ by bounding $\p{D_{n-n'} > 0}$. This is accomplished by the following proposition and its corollary. 
\begin{prop}
Let $H = |\{k \ge 0: S_k \ge 1\}|$. Then $\e{H} = \beta/(1-\beta)^2$. 
\end{prop}
\begin{proof}
This is an elementary fact about hitting times for simple random walk, and 
in particular follows from Examples~1.3.3 and~1.4.3 of \cite{norris98markov}.
\end{proof}
\begin{cor}
For $k \ge 0$, $\p{D_k>0} < \beta/(1-\beta)^2$. 
\end{cor}
\begin{proof}
For any $i \le k$ we have 
\[
\p{D_k>0, S_i=\min_{1 \le j \le k}S_j} \le \p{S_{k-i} \ge 1}\, ,
\]
and summing over $i$, plus a union bound, yields 
\[
\p{D_k>0} < \sum_{i \ge 0} \p{S_i \ge 1} = \e{H}\, . \qedhere
\]
\end{proof}
In reality, $Y(R_{n,p}^\cL[n'])$ may not equal zero, and so we should start the random walk $S$ not from zero but from a positive height. The following corollary addresses this.
\begin{cor}\label{cor:maxbound}
For any integers $d,k \ge 1$, 
\[\p{S_k+d > \min(M_k+d,0)} \le \p{S_k > -d} + \beta/(1-\beta)^2\, . \]
\end{cor}
\begin{proof}
Let $\tau = \inf\{i: S_i=-d\}$. By the Markov property, for any $i \le k$, $\p{S_k+d> \min(M_k+d,0)| \tau =i} = \p{D_{k-i}>0} < \beta/(1-\beta)^2$ by the preceding corollary. On the other hand, $\p{\tau > k} \le \p{S_k > -d}$, and the result follows. 
\end{proof}
On the other hand, if $t > Y(R_{n,p}^\cL[n'])$ then when the coupling succeeds we have 
$S_k+t - \min(M_k+t,0) \ge Y(R_{n,p}^\cL[n'+k])$ for all $0 \le k \le n-n'$. 
It then follows from Lemma~\ref{lem:coupling} and Corollary~\ref{cor:maxbound} that for any $\eps > 0$, 
\begin{align*}
\p{Y(R_{n,p}^\cL)>0} \nonumber 
\le &  \p{Y(R_{n,p}^\cL[n']) > \eps n} + \p{S_{n-n'} > -\eps n} + \frac{\beta}{(1-\beta)^2}+ O\left(n^{-\gamma/2}\right) \nonumber \\
\le & n^{-\Omega(n)} + e^{-\Omega(n)} + \frac{\beta}{(1-\beta)^2}+ O\left(n^{-\gamma/2}\right) \nonumber \\
=& O((\ln\ln n)^{-1/2})\, 
\end{align*}
where the second inequality follows from a Chernoff bound for $\p{S_{n-n'} > -\eps n}$ (assuming $\eps$ is chosen small enough), plus the bound from \refC{rank_firstminor}, and the last inequality follows from the definition of $\beta$ in Lemma~\ref{lem:coupling}. This proves the first assertion of Theorem~\ref{thm:gnpl}. 

When treating the symmetric model $Q_{n,p}^\cL$, the following modifications are required. First, Corollary~\ref{rank_firstminor} holds for all symmetric $n$-templates $\cL \in \cM^n(K)$ 
and with $Q_{n,p}^{\cL}[n']$ in place of $R_{n,p}^{\cL}[n']$. This was proved 
in \cite{costello08rank} for $Q_{n,p}[n']$, but as $\cL$ has size $K$, $|\rank(Q_{n,p}[n'])-\rank(Q_{n,p}^{\cL}[n'])|\le K^2=O(1)$, so the same bound holds for $Q_{n,p}^{\cL}[n']$. 

Second, we will likewise establish a coupling lemma for $Y(Q_{n,p}^\cL[m])$. 
\begin{lem}\label{lem:coupling2} 
For fixed $K \in \N$, there exists $C > 0$ such that the following holds. 
Given integer $n \ge 10$, let $\beta=\beta(n)=C(\ln\ln n)^{-1/4}$. 
Then uniformly over symmetric $\cL \in \cM^n(K)$ and $c\ln n/n \le p \le 1/2$, 
there exists a coupling of $(Y(Q_{n,p}^\cL[m]),n' \le m \le n)$ and a collection 
$(X_m,n'\le m< n)$ of iid random variables with $\p{X_i=1}=\beta$ and $\p{X_i=-1}=1-\beta$, 
such that with probability $1-O(n^{-\gamma/2})$, for all $n' \le m< n $, 
\[
Y((Q_{n,p}^\cL[m+1])- Y(Q_{n,p}^\cL[m]) \le \begin{cases}
					X_{m}	& \mbox{ if } Y(Q_{n,p}^\cL[m])>0 \\
					\max(X_{m},0)	& \mbox{ if } Y(Q_{n,p}^\cL[m])=0\, .
				\end{cases}
\]
\end{lem} 
Together, these two ingredients yield the second claim of Theorem~\ref{thm:gnpl} by a reprise of the arguments following Lemma~\ref{lem:coupling}
The remainder of the paper is therefore devoted to proving Lemmas~\ref{lem:coupling} and~\ref{lem:coupling2}.

\section{Rank increase via iterative exposure}\label{sec:rank}
In this section we focus on understanding when and why the rank increases. In what follows, fix an $m\times m$ matrix $Q=(q_{i,j})_{1\le i,j\le m}$. Given vectors $\xv=(x_1,\ldots,x_{m})$, $\yv=(y_1,\ldots,y_{m})$, we write 
\[
\Gamma(Q,\xv, \yv) = 
 \begin{pmatrix}
  q_{1,1} & \cdots  & q_{1,m} & y_1 \\
  \vdots  & \ddots  & \vdots & \vdots  \\
  \vdots  & \ddots  & \vdots & \vdots  \\
  q_{m,1} & \cdots & q_{m,m} & y_m \\
  x_{1} & \cdots & x_{m} & 0
 \end{pmatrix}
 \]

Now fix a matrix $Q$, and vectors $\xv=(x_1,\ldots,x_{m})$, $\yv=(y_1,\ldots,y_m)$. 
We remark that $\rank(\Gamma(Q,\xv,\yv))= \rank(Q)+2$ if and only if $\xv$ is linearly independent of the non-zero rows of $Q$ (i.e. it does not lie in the row-span of $Q$) and $\yv^T$ is linearly independent of the non-zero columns of $Q$. 
(In particular, if $Q$ is symmetric and $\xv=\yv$ then $\rank(\Gamma(Q,\xv,\yv))= \rank(Q)+2$ if and only if $\xv$ lies outside the row-span of $Q$.) Note that for this to occur $Q$ can not have full rank. 

We prove Lemmas~\ref{lem:coupling} and~\ref{lem:coupling2} as follows. First, we describe structural properties of $0-1$ matrices such that for any matrix $Q=(q_{ij})_{1 \le i,j \le m}$ satisfying such properties, for suitable random vectors $\xv$ and $\yv$, 
with high probability $\rank(\Gamma(Q,\xv,\yv))=\min(\rank(Q)+2,m+1)$. We then establish that with high probability, the matrices $(Q_{n,p}^{\cL}[m],n' \le m \le n)$ and $(R_{n,p}^{\cL}[m],n' \le m \le n)$ all have the requisite properties. 

More precisely, for {\em fixed} $Q$ and vectors $\xv,\yv$, we will see that $\rank(\Gamma(Q,\xv,\yv))=\min(\rank(Q)+2,m+1)$ if and only if a suitable linear, bilinear or quadratic form in $\xv$ and $\yv$, with coefficients determined by the matrix $Q$, vanishes; we elaborate on this very shortly. When $\xv$ and $\yv$ are Bernoulli random vectors, this leads us to evaluate the probability that a particular random sum is equal to zero. To bound such probabilities, we use {\em Littlewood--Offord bounds} proved in \cite{costello08rank},\cite{Costello12}, which we now state.
\begin{prop}\label{prop:LO_equations}
Let $x_1,\ldots,x_k, y_1, \ldots, y_k$ be independent Bernoulli$(p)$ random variables. \begin{itemize}
\item[(a)] Fix $a_1,\ldots,a_k \in \R\setminus \{0\}$. Then uniformly over $0 < p \le 1/2$, 
\[\sup_{r \in \R} \p{\sum_{i=1}^k a_ix_i = r} = O((kp)^{-1/2})\, .\]
\item[(b)] Fix $l\geq 1$ and $(a_{ij},1 \le i,j \le k)$ such that there are at least $l$ indices $j$ for which $|\{i: a_{i,j} \ne 0\}| \ge l$. 
Then uniformly over $0 < p \le 1/2$,
\[
\sup_{r \in \R} \p{\sum_{1 \le i,j \le k} a_{ij}x_iy_j = r} = O((lp)^{-1/2})\, .
\]
\item[(c)] With $l\geq 1$ and $(a_{ij},1 \le i,j \le k)$ as in (b), if also $a_{ij}=a_{ji}$ for all $1\leq i,j\leq k$, then uniformly over $0 < p \le 1/2$,
\[\sup_{r \in \R} \p{\sum_{1 \le i,j \le k} a_{ij}x_ix_j = r} = O((lp)^{-1/4})\, .\] 
\end{itemize}
\end{prop}
The matrix structural properties we require are precisely those that allow us to apply the bounds of Proposition~\ref{prop:LO_equations}. For this, the following definitions are germane. 
\begin{dfn}\label{dfn:blocked}
Fix a matrix $Q=(q_{ij})_{1 \le i,j \le m}$. 
\begin{itemize}
\item Given $S \subset [m]$, we say that $j \in [m]$ is an {\em $S$-selector} (for $Q$) 
if $|\{i \in S: q_{i,j} \ne 0\}|=1$. 
\item Given $2 \le b \le m$, we say $Q$ is {\em $b$-blocked} if any set $S \subset [m]$ with $S\cap Z^\textsc{row}(Q)=\emptyset$ and $2 \le |S| \le b$ has at least two $S$-selectors $j,l \in [m]$. 
\end{itemize}
\end{dfn}
The final condition in the definition says that in the sub-matrix formed by only looking at the rows in $S$, there are at least two columns containing exactly one non-zero entry. 
We call $j$ an $S$-selector as we think of $j$ as ``selecting'' the unique row $i$ with $q_{ij} \ne 0$. We remark that if a matrix $Q$ is $b$-blocked then any set $S$ of non-zero rows of $Q$ 
containing a linear dependency must have size at least $b+1$. More strongly, this is true even after deleting any single column of $Q$. 
\begin{dfn}\label{dfn:dense}
We say that $Q$ is {\em $b$-dense} if
\[
|\{i \in [m]:~\mbox{row $i$ of $Q$ has $> 1$ non-zero entry}\}| \ge b.
\]
\end{dfn}

We then have the following bounds, which are key to the proofs of Lemmas~\ref{lem:coupling} and~\ref{lem:coupling2}. 

\begin{prop}\label{prop:singular-block}
Fix integers $m \ge b \ge 1$ and a $b$-blocked $m\times m$ matrix $Q$ with $\rank(Q) < m-z^{\textsc{row}}(Q)$. 
Then uniformly over $0 < p \le 1/2$, if $\yv=(y_1,\ldots, y_m)$ has 
 iid Bernoulli($p$) entries then $\yv^T$ is independent of the non-zero columns of $Q$ with probability at least $1-O((bp)^{-1/2})$.
\end{prop}
\begin{proof}
Let $k=\rank(Q)$, and note that if $\rank(Q) < m-z^{\textsc{row}}(Q)$ then $1 \le k < m$. 
Write $\rv_1,\ldots,\rv_m$ for the rows of $Q$. 
By relabelling, we may assume that $\rv_1,\ldots,\rv_k$ are linearly independent and that 
$\rv_{k+1}$ is non-zero. 
It follows that there exist unique coefficients $a_1,\ldots,a_k$ for which $\rv_{k+1}=\sum_{i=1}^{k} a_i\rv_i$. Then $\{\rv_i: a_i \ne 0\} \cup \{\rv_{k+1}\}$ forms a set of linearly dependent non-zero rows, and so has size at least $b+1$ by the observation just after Definition~\ref{dfn:blocked}. 

Let $\hat{Q}$ be the matrix obtained from $Q$ by adding $\yv^T$ as column $m+1$. 
If $\yv^T$ lies in the column-span of $Q$ then $\rank(Q)=\rank(\hat{Q})$, so necessarily 
\[y_{k+1}= \sum_{i=1}^{k} a_i y_i.\]
Since $|\{\rv_i: a_i \ne 0\}| \ge b$, by Proposition \ref{prop:LO_equations} (a) we have
\[\p{\sum_{i=1}^{k} a_iy_i=0} = O((kp)^{-1/2}) =O((bp)^{-1/2}).\]
Therefore, the vector $\yv^T$ is independent of the non-zero columns of $Q$ with probability at least $1-O((bp)^{-1/2})$.
\end{proof}

\begin{prop}\label{prop:nonsing-block}
Fix integers $m \ge b \ge 1$ and a $b$-blocked, $b$-dense, $m\times m$ matrix $Q$ with $\rank(Q)=m$. Then uniformly over $0 < p \le 1/2$, if $\xv=(x_1,\ldots, x_m)$ and $\yv=(y_1,\ldots, y_m)$ have iid Bernoulli($p$) entries then 
$\p{\rank(\Gamma(Q,\xv,\yv))=m}=O((bp)^{-1/2})$.
\end{prop}
\begin{proof}
Let $A=(a_{ij})_{1\leq i,j\leq m}$ be the cofactor matrix of $Q$; that is, 
\[a_{ij}=(-1)^{i+j+1} det(Q^{(i,j)})\] 
where $Q^{(i,j)}$ is the $(i,j)$ minor of $Q$. 
A double-cofactor expansion of the determinant of $Q'=\Gamma(Q,\xv,\yv)$ yields 
\[det(Q')=\sum_{i,j=1}^m a_{ij}x_iy_j.\]
Note that $a_{ij}=0$ when $Q^{(i,j)}$ is singular, so we want to lower-bound the number of non-singular minors $Q^{(i,j)}$. 
To do so, fix $j\in [m]$ and write $Q^{(\emptyset,j)}$ for the $m\times m-1$ matrix obtained by deleting the $j$-th column of $Q$. 
Since $Q$ has full rank it has no zero rows. We claim that if $Q^{(\emptyset,j)}$ also has no zero rows then $|\{i \in [m]: a_{ij}\ne 0\}| > b$. To see this, note that $Q^{(\emptyset,j)}$  has rank $m-1$ and so since there are no zero rows, there exists (up to scaling factors) a unique vanishing linear combination
\[\sum_{i=1}^m c_i\rv_i^j=0,\]
where $\rv_i^j$ is the $i$'th row of $Q^{(\emptyset,j)}$.
Now, $Q^{(i,j)}$ is invertible (and thus $a_{ij} \ne 0$) if and only if $c_i\neq 0$. 
But the rows $\{\rv_i^j:c_i\neq 0\}$ are linearly dependent, and by the remark just after Definition~\ref{dfn:blocked}, since $Q$ is $b$-blocked we therefore have $|\{i \in [m]: c_i\ne 0\}| > b$. 

Finally, since $Q$ is $b$-dense, there are at most $m-b$ rows of $Q$ with exactly one non-zero entry. Thus, $|\{j \in [m]: Q^{(\emptyset,j)}\mbox{ has no zero rows}\}| \ge b$, 
and for any such $j$ we have $|\{i\in [m]\,: \, a_{ij}\neq 0\}|>b$ by the preceding paragraph. By \refP{prop:LO_equations} (b) it follows that, uniformly in $0 < p \le 1/2$, we have $\p{\det(Q')=0}\le O((bp)^{-1/2})$ as claimed. 
\end{proof}

The following proposition is an analogue of Proposition~\ref{prop:nonsing-block} which we use in analyzing the symmetric Bernoulli process. 
\begin{prop}\label{prop:nonsing-block_sym}
Fix integers $m \ge b \ge 1$ and a $b$-blocked, $b$-dense, $m\times m$ symmetric matrix $Q$ with $\rank(Q)=m$. Then uniformly over $0 < p \le 1/2$, if $\xv=(x_1,\ldots, x_m)$ has iid Bernoulli($p$) entries then 
$\p{\rank(\Gamma(Q,\xv,\xv))=m}=O((bp)^{-1/4})$.
\end{prop}
\begin{proof}[Proof sketch.]
The proof is nearly identical to that of Proposition~\ref{prop:nonsing-block}.
However, in this case the double cofactor expansion of $\det(\Gamma(Q,\xv,\xv))$ has the form $\sum_{i,j=1}^m a_{ij} x_ix_j$. Consequently, we conclude by applying part (c), rather than part (b), of Proposition~\ref{prop:LO_equations}. We omit the details. 
\end{proof}

We will apply Propositions~\ref{prop:singular-block} and~\ref{prop:nonsing-block} via the following lemma. 
\begin{lem}\label{lem:optimal_increase}
Fix integers $m \ge b \ge 1$ and an $m\times m$ matrix $Q$ for which both $Q$ and $Q^T$ are $b$-blocked and $b$-dense. Then uniformly over $0 < p \le 1/2$, if $\xv=(x_1,\ldots, x_m)$ and $\yv=(y_1,\ldots, y_m)$ have iid Bernoulli($p$) entries then 
\begin{align*}
&\quad \p{\rank(\Gamma(Q,\xv,\yv)) \ne \rank(Q)+1 +\I{Y(Q) >0\}}} \\
= &\quad O((bp)^{-1/2}).
\end{align*}
\end{lem}
\begin{proof}
In what follows we write $Q'=\Gamma(Q,\xv,\yv)$. Recall that if $\xv$ and $\yv$ lie outside the row-span and column-span of $Q$, respectively, then $\rank(Q')=\rank(Q)+2$. Note also that $Y(Q)=Y(Q^T)$ always holds. 

If $Y(Q)> 0$ then by the definition of $Y(Q)$ we have $\rank(Q) < m-z^{\textsc{row}}(Q)$ and 
\[
\rank(Q^T)=\rank(Q) < m-z^{\textsc{col}}(Q)=m-z^{\textsc{row}}(Q^T).
\]
In this case the lemma follows by applying Proposition~\ref{prop:singular-block} twice, once to $Q$ and $\yv$ and once to $Q^T$ and $\xv$. 

We now treat the case $Y(Q)=0=Y(Q^T)$. By replacing $Q$ by $Q^T$ if necessary, 
we may assume that $Q$ has $s$ non-zero rows and $t$ non-zero columns, for some $0 \le t \le s \le m$; in particular note that $\rank(Q)=t$. By relabelling the rows and columns, we may assume that $Q'$ has the form 
\[Q'=\begin{pmatrix}
  A  & \Zv  &  (\yv')^T \\
  \Zv & \Zv  &  (\yv^+)^T  \\
  \xv' & \xv^-  & 0
 \end{pmatrix}\, ,\]
 where $A$ is an $s\times  t$ matrix with no zero rows or columns, $\Zv$ represents a block of zeros, and where $\xv=(\xv',\xv^-)$ and $\yv=(\yv',\yv^+)$. 
 
 If $t=s$ then $A$ is $b$-blocked and $b$-dense and $\rank(A)=t=\rank(Q)$. Since 
 \[
 \rank(Q') \ge \rank(\Gamma(A,\xv',\yv'))
 \]
and $\xv',\yv'$ have iid Bernoulli$(p)$ entries, in this case the lemma follows by applying
\refP{prop:nonsing-block} to $A$, $\xv'$ and $\yv'$. 

Finally, if $t <s $ then $\rank(Q)=t < s= m-z^{\textsc{row}}(Q)$. 
Proposition~\ref{prop:singular-block} applied to $Q$ and $\yv$ then yields that $\yv$ lies outside the column-span of $Q$ with probability $1-O((bp)^{-1/2})$. If the latter occurs then $\rank(Q') \ge \rank(Q)+1$; this completes the proof. 
\end{proof}
The analogous result for symmetric matrices is as follows. 
\begin{lem}\label{lem:optimal_increase_sym}
Fix integers $m \ge b \ge 1$ and a $b$-blocked, $b$-dense symmetric $m\times m$ matrix $Q$. Then uniformly over $0 < p \le 1/2$, if $\xv=(x_1,\ldots, x_m)$ has iid Bernoulli($p$) entries then 
\begin{align*}
&\quad \p{\rank(\Gamma(Q,\xv,\xv)) < \rank(Q)+1 +\I{Y(Q) >0\}}} \\
= &\quad O((bp)^{-1/4}).
\end{align*}
\end{lem}
The proof is practically identical to that of \refL{lem:optimal_increase}, but is slightly easier as for symmetric matrices we always have $z(Q)=z^{\textsc{row}}(Q)=z^{\textsc{col}}(Q)$. The resulting bound is weaker as we must use Proposition~\ref{prop:nonsing-block_sym} rather than \refP{prop:nonsing-block}. We omit the details. 

To shorten coming formulas, we introduce the following shorthand. 
\begin{dfn}\label{dfn:robust}
For $n \ge 1$, let $k=k(n,p)=\ln\ln n/(2p)$. We say that a square matrix $Q$ is {\em $n$-robust} if both $Q$ and $Q^T$ are $k$-blocked and $k$-dense. 
\end{dfn}

The following proposition, whose proof is the most technical part of the paper, says that robustness is very likely to hold throughout the final $n-n'$ steps of the iterative exposure of minors in $R_{n,p}^{\cL}$. In the following proposition, recall from the start of Section~\ref{sec:iterative} the definition of permissible templates, and also the fact that $\gamma\in (0,1/4)$ is a fixed constant depending only on $c$.

\begin{prop}\label{prop:key}
Fix $K \in \N$. For any $p \in (c\ln n/n,1/2)$ and any permissible template $\cL\in \cM^n(K)$, 
we have 
\begin{align*}
 \p{\forall~m \in [n',n]\,: R_{n,p}^\cL[m]~\mbox{is}~n\mbox{-robust}} & =  1-O\left(n^{-\gamma}\right)\, ,\mbox{ and} \\
  \p{\forall~m \in [n',n]\,: Q_{n,p}^\cL[m]~\mbox{is}~n\mbox{-robust}} & =  1-O\left(n^{-\gamma}\right). 
\end{align*}
\end{prop}
We provide the proof of \refP{prop:key} in Section~\ref{sec:sketchproof}, for now using it to complete the proofs of \refL{lem:coupling} and \refL{lem:coupling2} (and so of Theorem~\ref{thm:gnpl}). 
We begin by controlling the probability that $z(R_{n,p}^{\cL}[m])$ ever decreases by more than one in a single step of the minor exposure process. 
\begin{lem}\label{lem:one isolate} For fixed $K \in \N$, uniformly over permissible $\cL \in \cM^n(K)$ and $c\ln n/n \le p \le 1/2$ we have 
\begin{align*}
\p{\exists n'\le m< n\,: z(R_{n,p}^\cL[m+1])<z(R_{n,p}^\cL[m])-1}&= O(n^{-\gamma/2}).
\end{align*}
\end{lem}

\begin{proof}
First, by symmetry this probability is at most twice 
\[
\p{\exists n' \le m \le n\,: |Z^{\textsc{row}}(R_{n,p}^\cL[m])|-|Z^{\textsc{row}}(R_{n,p}^\cL[m+1])|>1\,}.
\]
For $p\ge \frac{20\ln n}{n}$, with high probability $|Z^{\textsc{row}}(R_{n,p}^\cL[m])|=0$ for each $n'\le m\le n$, 
so we assume that $\frac{c \ln n}{n}\le p\le \frac{20\ln n}{n}$.
The matrix $R_{n,p}[n']$ is distributed as $R_{n',p}$, and we have $p\ge \alpha c\ln n'/n' = (1/2+\gamma) \ln n'/n'$. 
Since $R_{n,p}^\cL$ contains at most $K^2$ rows with deterministic or partially deterministic coordinates, it follows that for $n$ large, 
\begin{align}
&\p{|Z^{\textsc{row}}(R_{n,p}^\cL[n'])| \ge n^{1/2-\gamma/2}} \nonumber \\ 
\le & {n' \choose n^{1/2-\gamma/2}} \left((1-p)^{n'-K^2}\right)^{n^{1/2-\gamma/2}} \nonumber \\
\le & e^{-n^{\gamma/2}}. \label{eq:jbound}
\end{align}
We next bound the probability that $z=|Z^{\textsc{row}}(R_{n,p}^\cL[n'])|<n^{1/2-\gamma/2}$ and at least two zero rows disappear in a single step. For fixed $m$ with $n' \le m < n$ we have 
\begin{align}
&\p{|Z^{\textsc{row}}(R_{n,p}^\cL[m+1])|< |Z^{\textsc{row}}(R_{n,p}^\cL[m])|-1, z< n^{1/2-\gamma/2}} \nonumber \\ 
\le& {\lfloor n^{1/2-\gamma/2}\rfloor \choose 2} p^2 \, ,\label{eq:twodownbound}
\end{align}
which is at most $n^{-1-\gamma/2}$ for $n$ large and $p \le 20\ln n/n$. By \eqref{eq:jbound}, \eqref{eq:twodownbound}, and a union bound, 
the result follows. 
\end{proof}
We state the symmetric analogue of Lemma~\ref{lem:one isolate} for later use. 
\begin{lem}\label{lem:one isolate_sym}
Under the conditions of Lemma~\ref{lem:one isolate}, if $\cL$ is symmetric then 
\[
\p{\exists n'\le m\le n\,: z(Q_{n,p}^\cL[m])-z(Q_{n,p}^\cL[m+1])>1}\leq O(n^{-\gamma/2}).
\]
\end{lem}
\begin{proof}
In this case, the desired probability is equal to the probability that
\[|Z(Q_{n,p}^\cL[m])|-|Z(Q_{n,p}^\cL[m+1])|>1\]
for some $n'\le m\le n$. The proof then follows as that of \refL{lem:one isolate}. 
\end{proof}

\begin{proof}[Proof of \refL{lem:coupling}]
For $n' \le m \le n$ let $\cF_{n,m} = \sigma(\{R_{n,p}^\cL[i]\,:n' \le i \le m\})$ and 
let $E_{n,m} = \{\forall n' \le i \le m\,: R_{n,p}^\cL[i]~\mbox{is}~n\mbox{-robust}\}$. 
Note that $E_{n,m} \in \cF_{n,m}$ for all $n' \le m \le n$. Also, $E_{n,j} \subset E_{n,i}$ for all $n' \le i < j \le n$. 

Now for $n' \le m \le n-1$ let 
$C_m = \{\rank(R_{n,p}^\cL[m+1]) = \rank(R_{n,p}^\cL[m])+1+\I{Y(R_{n,p}^\cL[m])>0}\}.$
Then since $R_{n,p}^\cL[m]$ is $\cF_{n,m}$-measurable and $E_{n,m} \in \cF_{n,m}$, we have 
\begin{align*}
	& \p{C_m~|~\cF_{n,m}} \\
\ge & \p{C_m~|~\cF_{n,m}} \I{E_{n,m}}\\
\ge & 	\inf\{\p{C_m | R_{n,p}^{\cL}[m]=Q}: Q~\mathrm{is}~n\textrm{-robust}\}\cdot\I{E_{n,m}}. \\
\ge& (1-O((kp)^{-1/2}))\I{E_{n,n}}\, .
\end{align*}
the last inequality by \refL{lem:optimal_increase} and since $\I{E_{n,m}} \ge \I{E_{n,n}}$. 
Therefore, there exists $C> 0$ such that 
for all $n' \le m < n$, writing $\beta=C(\ln \ln n)^{1/2}$, we have 
\begin{align*}
 & \p{\rank(R_{n,p}^{\cL}[m+1])< \rank(R_{n,p}^{\cL}[m])+1+\I{Y_m>0}~|~\cF_{n,m}} \\
 \le&  C(2kp)^{-1/2} \I{E_{n,n}} + \I{E_{n,n}^c} \\
\le & \beta +\I{E_{n,n}^c}. 
\end{align*}
For $n' \le m < n$ let $I_{m} = \I{\rank(R_{n,p}^{\cL}[m+1])< \rank(R_{n,p}^{\cL}[m])+1+\I{Y_m>0}}$. 
It follows from the preceding bound that for $m \in [n',n-1]$, 
\[
\p{I_m = 1|I_{n'},\ldots,I_{m-1},E_{n,n}} \le \beta. 
\]
We may therefore couple $(I_m,n' \le m < n)$ with a family $(B_m,n'\le m < n)$ of independent Bernoulli$(\beta)$ random variables 
such that for all $n'< m \le n$, 
\[
I_m \le B_m + (1-B_m)\I{E_{n,n}^c}.
\]
Finally, for $n' \le m < n$ let $X_m=2B_m-1$, so that $\p{X_m=1}=\beta=1-\p{X_m=-1}$. 
By the identity (\ref{eq:ychange}) for $Y(R_{n,p}[m+1])$, if $Y(R_{n,p}[m+1])-Y(R_{n,p}[m]) \le \max(X_{m+1},-Y(R_{n,p}[m]))$ then either $I_m > B_m$ (in which case $E_{n,n}^c$ occurs) or 
$\{z(R_{n,p}^{\cL}[m+1]) \le z(R_{n,p}^{\cL}[m])-2\}$. 
It follows that 
\begin{align*}
& \p{\forall n' \le m < n\,: Y(R_{n,p}^\cL[m+1])-Y(R_{n,p}^\cL[m]) \le \max(X_{m+1},-Y(R_{n,p}[m]))}\\
\ge&  1-\p{E_{n,n}^c}-\p{\exists n' \le m < n: z(R_{n,p}^{\cL}[m+1]) \le z(R_{n,p}^{\cL}[m])-2} \\
= & 1-O(n^{-\gamma/2})\, ,
\end{align*}
the final bound by \refL{lem:one isolate} and \refP{prop:key}. This completes the proof. 
\end{proof}

The proof of Lemma~\ref{lem:coupling2} is practically identical, using the second rather than the first bound of Proposition \ref{prop:key} and using Lemmas~\ref{lem:one isolate_sym} and~\ref{lem:optimal_increase_sym} rather than Lemmas~\ref{lem:one isolate} and~\ref{lem:optimal_increase}, respectively. We omit the details. 

\section{Structural properties that guarantee rank increase}
\label{sec:structure}

In this section we prove \refP{prop:key}. For the remainder of the paper, fix $K \in \N$, $n\in \N$ large, let $k=k(n,p)=\ln \ln n/(2p)$, and fix a permissible template $\cL = (\cL^+,\cL^-)=((S_i^+)_{i \in I^+},(S_j^-)_{j \in I^-}) \in \cM^n(K)$. For $i \in [n]$ write $R_i = R_{n,p}^{\cL}[i]$, $H_i=H_{n,p}^{\cL}[i]$. Also for the remainder of the paper, 
$T = [n] \setminus (I^- \cup \bigcup_{i \in I^+} S_i^+)$ and let $U=\bigcup_{i \in I^+} S_i^+$.
These definitions are illustrated in Figure~\ref{fig:1}. Finally, recall that $c \in (1/2,1)$ and $\alpha$ are fixed so that $\alpha c \in (1/2,3/4)$, that $\gamma=\alpha c-1/2$ and that $n'=\lceil \alpha n\rceil$. 

\begin{figure}[htb]
\begin{center}
\begin{picture}(310,210)
\put(0,0){\includegraphics[scale=1]{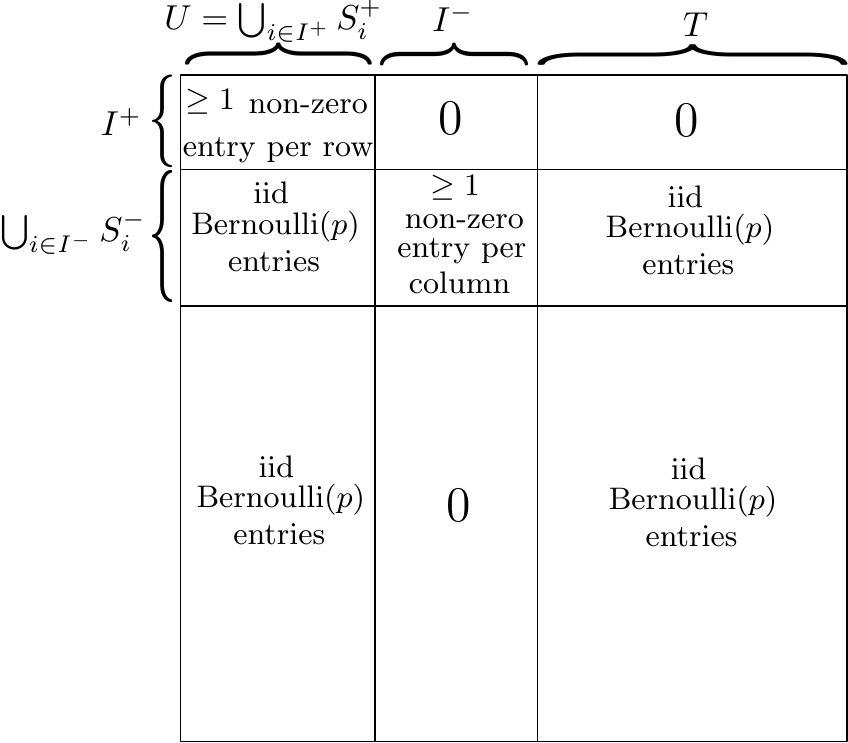}}
\end{picture}
\end{center}
\caption{The deterministic and random structure of the matrix $R_n$.}
\label{fig:1}
\end{figure}
Before proceeding to details, we pause to describe the broad strokes of our proof. Our arguments are more straightforwardly described in the language of graphs rather than matrices, so we shall begin to switch to the language of graphs. 

We separately bound the probability that for some $n'\le m\le n$, $R_m$ either is not $k(n,p)$-blocked or is not $k(n,p)$-dense. Bounding the latter probability is straightforward: this is essentially the event there are too many vertices with low out-degree in $H_{n'}$, and $p\ge c\ln n/n$ is large enough that such vertices are rare. 

Bounding the probability that $R_m$ is not $k(n,p)$-blocked for some $m$ is more involved, and we pause to develop some intuition. Recall that for $R_m$ to be $k(n,p)$-blocked, we need that for any $S \subset [m]$ with $2 \le |S| \le k(n,p)$, there are at least two $S$-selectors in $R_m$. In the language of graphs, an $S$-selector is a vertex $v$ such that $v$ has exactly one in-neighbour in $S$. For $n' \le m \le n$ and $i \in S$, conditional on $\{N^+_{H_m}(j):j \in S,j \ne i\}$, the larger the degree of $i$ the more likely it is that $N^+_{H_m}(i)$ contains a vertex $v$ lying outside $\bigcup_{j\in S\setminus\{i\}} N^+_{H_m}(j)$, and such a vertex $v$ is an $S$-selector. For this reason, low-degree vertices pose a potential threat to the existence of $S$-selectors. (Indeed, low-degree vertices in a sense pose the greatest difficulty for the proof; it is precisely out-degree one vertices that cause Theorems~\ref{thm:digraph} and Theorem~\ref{thm:gnpl} to be fail for $c < 1/2$.) We neutralize this threat by showing that with high probability all (sufficiently) low degree vertices have pairwise disjoint out-neighbourhoods, and so sets $S$ of {\em exclusively} low-degree vertices have many $S$-selectors. 

We now turn to details. We begin by bounding the probability that some $R_m$ is not $k(n,p)$ dense. 
\begin{lem}\label{lem:dense}
Uniformly over $c\ln n/n\le p\le 1/2$ 
we have 
\[
 \p{\forall~n'\le m \le n: R_m~\mbox{is}~k(n,p)\mbox{-dense}} = 1-O(n^{-1})\, ,
\]
and if $\cL$ is symmetric then 
\[
 \p{\forall~n'\le m \le n: Q_{n,p}^\cL[m]~\mbox{is}~k(n,p)\mbox{-dense}} = 1-O(n^{-1})\, 
\]
\end{lem}

\begin{proof}
For $n'\le m\le n$, let 
\[A_m=\{i\in [n']\setminus I^+\,: |([m]\cap N_{R_m}^+(i)) \setminus I^- |> 1 \}. \]
Observe that $A_{n'}\subset A_m$ for all $n'\le m\le n$. On the other hand, if $R_m$ is $k(n,p)$-dense, then $|A_m|\ge k(n,p)-K(K+1) \ge k(n,p)/2$, the latter for $n$ large. It follows that
\begin{align*}
\p{\exists~n'\le m \le n: R_{n,p}^\cL[m]~\mbox{is not}~k(n,p)\mbox{-dense}} \le \p{|A_{n'}|< k(n,p)/2}.
\end{align*}
For $i\in [n']\setminus (I^+\cup \bigcup_{i \in I^-}S_i^-)$ let $B_i$ be the event that $i\notin A_{n'}$. The event $B_i$ is monotone decreasing, 
so in bounding its probability from above we may assume that $p$ is equal to $p_{\mathrm{min}} =c\ln n/n$. 
Write $s=|I^-|\le K$. For $i \in [n']\setminus (I^+\cup \bigcup_{i \in I^-}S_i^-)$ we then have 
\begin{align*}
&   \quad\p{B_i} \\
= & \quad  \p{\mathrm{Binomial}(n'-s,p_{\mathrm{min}}) \le 1} \\
= & \quad\left(1+\frac{(n'-s)p_{\mathrm{min}}}{1-p_{\mathrm{min}}}\right)(1-p_{\mathrm{min}})^{n'-s} \\
\le & \quad  2(1+\alpha c\ln n)n^{-\alpha c},
\end{align*}
so 
$ \E{|[n']\setminus A_{n'}|} 
\le n'\cdot 2(1+\alpha c\ln n)n^{-\alpha c} 
\le 2n^{1-\alpha c}(1+\alpha c \ln n). $
A similar calculation shows that $\E{|[n']\setminus A_{n'}|^4}\le (2n^{1-\alpha c}(1+\alpha \ln n))^4 16(1+\alpha c\ln n)$ and so by Markov's inequality 
\[
\p{|A_{n'}|< k(n,p)/2} \le \frac{\E{|[n']\setminus A_{n'}|^4}}{(n'-k(n,p)/2)^4} 
\le \frac{(4n^{1-\alpha c})^4(1+\alpha \ln n)^5}{(n'-k(n,p)/2)^4} 
<  \frac{1}{n^2}\, , 
\]
the last inequality holding for $n$ large since $1-\alpha c<1/2$ and $n'-k(n,p)/2=\Omega(n)$. The lemma follows by a union bound. An identical proof establishes the stated bound for $Q_{n,p}^\cL$ in the case that 
$\cL$ is symmetric. 
\end{proof}
Next, we address the probability that the minors $R_m$, $n'\le m\le n$,
are not $k(n,p)$-blocked. The next definition allows us to avoid the (partially) deterministic neighbourhoods of $H_m$. 

\begin{dfn}\label{dfn:blocked2}
Fix $b \ge 2$. A matrix $Q=(q_{ij})_{1\le i,j\le m}$ is {\em $(b,\cL)$-blocked} if any set $S \subset [m]$ with $2 \le |S| \le b$ satisfies the following conditions. 
\begin{itemize} 
\item If $S\cap (Z^{\textsc{row}}(Q)\cup I^+)=\emptyset$ then there exist distinct 
$j,l \in [m]\setminus (I^- \bigcup_{i \in I^+} S_i^+)$ that are $S$-selectors for $Q$. 
\item If $S\cap (Z^{\textsc{col}}(Q)\cup I^-)=\emptyset$ then there exist distinct 
$j,l \in [m]\setminus (I^+ \bigcup_{i \in I^-} S_i^-)$ 
that are $S$-selectors for $Q^T$. 
\end{itemize}
\end{dfn}
(In the last bullet of the preceding definition, it may be useful to note that if $M$ is the adjacency matrix of a directed graph then $M^T$ is the adjacency matrix of the graph with all edge orientations reversed.) We then have the following lemma. 
\begin{prop}\label{lem:key}
Uniformly in 
$c\ln n/n \le p\le 1/2$, 
\[
\p{\forall n'\le m\le n: R_m \text{ is $(k,\cL)$-blocked}} =1- O\left(n^{-\gamma}\right)\, ,
\]
and if $\cL$ is symmetric then also 
\[
\p{\forall n'\le m\le n: Q_{n,p}^{\cL}[m] \text{ is $(k,\cL)$-blocked}} =1- O\left(n^{-\gamma}\right)\, .
\]
\end{prop}
The proof of \refP{prop:key} assuming \refP{lem:key} is straightforward and largely consists of showing that {\em if} $R_m$ is $(k,\cL)$-blocked then most sets $S \in [m]\setminus Z^{\textsc{row}}(R_m)$ {\em deterministically} have at least two $S$-selectors (even if $S \cap I^+ \ne \emptyset$). An easy probability bound then polishes off the proof. 
\begin{proof}[Proof of \refP{prop:key}]
Let 
\begin{align*}
C_1^m & = \{E \subset [m]\setminus Z^{\textsc{row}}(R_m): 2 \le |E| \le k(n,p), E \subset I^+\}\, ,\\
C_2^m & = \{E \subset [m]\setminus Z^{\textsc{row}}(R_m): 2 \le |E| \le k(n,p), |E \setminus I^+| \ge 2\}\, ,\\
C_3^m & = \{E \subset [m]\setminus Z^{\textsc{row}}(R_m): 2 \le |E| \le k(n,p), |E \setminus I^+|=1\}\, ,
\end{align*}
and for $k=1,2,3$ let $A_k^m = \{\forall E \in C_k^m:~\mbox{there are two $E$-selectors in}~T\cap [m]\}$. Note that if $R_m$ is $(k,\cL)$-blocked for all $n' \le m \le n$ but is not $k$-blocked for some $n' \le m \le n$, then one of the events $A_k^m$, $k \in \{1,2,3\}$, $n' \le m \le n$ must fail to occur. We consider the events $A_k^m$, $k=1,2,3$ in turn. 

First, note that since the sets $(S_i^+,i \in I^+)$ are disjoint and non-empty, for every $E \in C_1^m$, for all $i \in E$, every $j \in S_i^+$ is an $E$-selector. Thus, $A_1^m$ holds deterministically. 

Second, if $R_m$ is $(k,\cL)$-blocked then for any $E \in C_2^m$ there are $(E\setminus I^+)$-selectors $\ell_1,\ell_2 \in T\cap [m]$. Since $\bigcup_{i \in I^+} S_i^+$ is disjoint from $T$, it follows that $\ell_1,\ell_2$ are not in the out-neighbourhoods of any vertex in $I^+$. Therefore $\ell_1$ and $\ell_2$ are also $E$-selectors, so if $R_m$ is $(k,\cL)$-blocked then $A_2^m$ holds. It follows by Proposition~\ref{lem:key} that $\p{\bigcap_{m=n'}^n A_2^m} = 1-O(n^{-\gamma})$. 

Third, fix $E \in C_3^m$ and write $E \setminus I^+ = \{v\}$. 
Note that $v$ must have at least one neighbour in $H_m$ as $E \cap Z^{\textsc{row}}(R_m)=\emptyset$. If $|N_{H_m}^+(v) \cap T| \ge 2$ then any two $\ell_1,\ell_2 \in N_{H_m}^+(v) \cap T$ are $E$-selectors
since $N_{H_m}^+(I^+) \cap T = \emptyset$. Also, if $|N_{H_m}^+(v) \cap T| \ge 1$ and $N_{H_m}^+(v) \cap N_{H_m}^+(I^+) = \emptyset$ then choose $\ell_1 \in S_i^+\subset N_{H_m}^+(I^+)$ for some $i \in E\cap I^+$, and $\ell_2 \in N_{H_m}^+(v) \cap T$; both $\ell_1$ and $\ell_2$ are again $E$-selectors. 

It follows that 
\begin{equation}
\p{\bigcup_{m=n'}^n (A_3^m)^c} \le 
\p{\exists v \in [n] \setminus I^+: |N_{H_n}^+(v)\cap U| \ge 1, | N_{H_{n}}^+(v)\cap T\cap [n']|\le 1}\, .\label{eq:toprove}
\end{equation} 
For fixed $v \in [n]\setminus I^+$, since $|[n']\cap T|\ge n'-K(K+1)$, we have 
\begin{align*}
 \p{|N_{H_{n'}}^+(v)\cap T\cap [n']|\le 1} 
\le 	& \p{\mathrm{Binomial}(n'-K(K+1)-1,p)\le 1} \\
\le & (1+n'p)(1-p)^{n'-(K+1)^2} 
\end{align*}
Furthermore, 
\[
\p{|N_{H_n}^+(v)\cap U| \ge 1} \le K^2 p. 
\]
The events in the two preceding probabilities are independent since $U$ and $T$ are disjoint. By a union bound over $v \in [n]\setminus I^+$, it follows that the probability in (\ref{eq:toprove}) is bounded by 
\[
n \cdot  (1+n'p)(1-p)^{n'-(K+1)^2} \cdot K^2 p. 
\]
If $p \ge 4\ln n/(\alpha n)$ then $(1-p)^{n'-(K+1)^2} \le e^{-p(n' -(K+1)^2)} \le n^{-3}$ for $n$ large, proving the result in this case. If $p \le 4\ln n/(\alpha n)$ then 
$np \le (4/\alpha)\ln n$, and since $p \ge c\ln n/n$ and $n' \ge \alpha n$, this expression is bounded by 
\[
K^2(1+(4/\alpha)\ln n)^2 (1-p)^{-(K+1)^2} e^{-pn'} = O(\ln^2 n/n^{\alpha c})\, .
\]
Since $\alpha c = \gamma+1/2$ we conclude that 
\begin{equation}
\p{\exists v \in [n] \setminus I^+: |N_{H_n}^+(v)\cap U| \ge 1, | N_{H_{n}}^+(v)\cap T\cap [n']|\le 1} = O(n^{-\gamma})\, .\label{eq:use later}
\end{equation}
Combining this bound with our bound on $\p{\bigcap_{m=n'}^n A_2^m}$ and our deterministic observation about the events $A_1^m$, it follows that 
\[
\p{\forall n'\le m\le n: R_m \text{ is $k$-blocked}}= 1-O(n^{-\gamma})\,.
\]
A symmetric argument for $R_m^T$ (using the second bullet from Definition~\ref{dfn:blocked2} instead of the first) shows that 
\[
\p{\forall n'\le m\le n: R_m^T \text{ is $k$-blocked}}= 1-O(n^{-\gamma})\, ,
\]
which completes the proof of the first assertion of the Proposition~\ref{prop:key}. 
The second assertion of the proposition follows by a practically identical argument using the second bound of Proposition~\ref{lem:key} (the only difference is that in this case there is no need to conclude by ``a symmetric argument'' as $Q_{n,p}^{\cL}[m] =(Q_{n,p}^{\cL}[m])^T $). 
\end{proof}

The remainder of the paper is devoted to the proof of Proposition~\ref{lem:key}. 
\subsection{Proof of Proposition~\ref{lem:key}}\label{sec:sketchproof}
First, suppose $\cL$ is symmetric and 
write $D=I^- \cup \bigcup_{i \in I^+} S_i^+=I^+ \cup \bigcup_{i \in I^-} S_i^-$. Then for $n' \le m \le n$, $Q_{n,p}^\cL[m]$ is $(k,\cL)$-blocked if and only if $Q_{n,p}[[m]\setminus D]$ is $k$-blocked. It follows that 
\begin{align*}
& \quad \p{\forall n'\le m\le n: Q_{n,p}^{\cL}[m] \text{ is $(k,\cL)$-blocked}}\\
= & \quad \p{\forall n'\le m\le n: Q_{n,p}[[m]\setminus D] \text{ is $k$-blocked}} \\
= & \quad 1-O(n^{-\gamma})\, ,
\end{align*}
the last bound by Lemma~2.10 of~\cite{costello08rank} (note that in that paper, the first property in the definition of ``good'' is equivalent to our property ``$k$-blocked''). This establishes the second assertion of the lemma, so  
we may now focus exclusively on the first. 

We say a set $E \subset [m]\setminus I^+$ is {\em blocked} if $T \cap [m]$ contains two distinct 
$E$-selectors. 
Given $n' \le m \le n$ and $2 \le s \le k(n,p)$, let 
\[
D_{m,s} = \{\exists E \subset [m]\setminus (Z^{\textsc{row}}(R_m) \cup I^+): |E|=s,E\mbox{ is not blocked}\}\, . 
\]
To prove 
Proposition~\ref{lem:key} it suffices to show that 
\begin{equation}\label{eq:key_toprove}
\p{\bigcup_{m=n'}^n \bigcup_{s=2}^{k(n,p)}D_{m,s}} = O(n^{-\gamma}). 
\end{equation}
Since $n'=\Theta(n)$, our arguments are mostly insensitive to the value of $m \in [n',n]$. The value of $s$ plays a more significant role, and we tailor our arguments for different values. 

The region where $(p\ln^{1/2} n)^{-1} \le s \le k(n,p)$ is rather straightforward; fix such $s$ and $n' \le m \le n$, and fix $E \subset [m]\setminus (Z^{\textsc{row}}(R_m) \cup I^+)$ with $|E|=s$. 
Then for fixed $j \in [m]\setminus (I^- \cup \bigcup_{i \in I^+} S_i^+)$, 
\[
\p{j\mbox{ is an $E$-selector}} = \p{|N^-_{R_m}(j) \cap E|=1}= sp(1-p)^{s-1}\ge spe^{-sp}\, .
\]
These events are independent for distinct $j \in [m]\setminus (I^- \cup \bigcup_{i \in I^+} S_i^+)$, 
and it follows that 
\[
\p{E~\mbox{is not blocked}} \le \p{\mathrm{Bin}(m-K(K+1),spe^{-sp}) \le 1} \le n(1-spe^{-sp})^{n/2}\, ,
\]
the last inequality since $m-K(K+1) \ge n/2$ for $n$ large. Since ${n \choose s} \le \exp(s\ln (ne/s))$, it follows 
by a union bound over $E \subset [m]\setminus I^+$ that 
\begin{align}
\p{D_{m,s}} 	& \le \exp(s\ln(ne/s)+\ln n) (1-spe^{-sp})^{n/2} \nonumber\\
			& \le \exp(s\ln(ne/s)+\ln n-nspe^{-sp}/2)\, \nonumber\\
			& = \exp(\ln n + s(\ln(ne/s)-npe^{-sp}/2)) \nonumber\\
			& \le \exp(\ln n + s(\ln(npe\ln^{1/2} n) - np/\ln^{1/2}n))\, , \label{eq:intermediate}
\end{align}
the last bound following since $(p\ln^{1/2}n)^{-1} \le s \le \ln\ln n/(2p)=k(n,p)$. 
Using that $x/y \ge 2\ln(xy)$ when $x \ge y^2/2 \ge 2e^6$, since $np \ge c\ln n \ge (\ln^{1/2} n)^2/2$ it follows that 
$np/\ln^{1/2}n \ge 2 \ln(npe\ln^{1/2} n)$ for $n$ sufficiently large, and so (\ref{eq:intermediate}) yields 
\[
\p{D_{m,s}} \le \exp(\ln n - snp/(2\ln^{1/2} n)) \le \exp(\ln n - n/(2\ln n))\, ,
\]
where in the final inequality we use that $s \ge (p\ln^{1/2} n)^{-1}$. 
A union bound and the fact that $(n-n'+1)(\ln \ln n/(2p)-(p\ln^{1/2} n)^{-1}) \le n^2$ then yields 
\begin{equation}\label{eq:slarge}
\p{\bigcup_{n' \le m \le n} \bigcup_{(p\ln^{1/2} n)^{-1} \le s \le \ln \ln n/(2p)} D_{m,s}} \le \exp(3\ln n - n/(2\ln n))\, .
\end{equation}
This takes care of the range $(p\ln^{1/2}n)^{-1} \le s \le k(n,p)$, which for small $p$ is the lion's share of the values of $s$ under consideration (though the smaller values of $s$ require slightly more work). 

For $2 \le s \le 1/(p\ln^{1/2} n)$, our approach to bounding $\p{D_{m,s}}$ is based on the pigeonhole principle and a simple stochastic relation, and we now explain both. 
For convenience set $\hat{n}=n' - K(K+1)$, and note that $|T \cap [m]| \ge \hat{n}$. 
Note that for a set $E \subset [m]\setminus (Z^{\text{row}}(H_m)\cup I^+)$ 
if $E$ is not blocked then at most one vertex in $N_{H_m}^+(E)\cap T$ has less than two in-neighbours in $E$, so $\sum_{i \in E} |N_{H_m}^+(i) \cap T| \ge 2 |N_{H_m}^+(E)\cap T| - 1$. 
It follows that 
\begin{equation}\label{pigeonhole}
\p{E~\mbox{is not blocked}} \le \p{|N_{H_m}^+(E)\cap T| \le \frac{\sum_{i \in E} |N_{H_m}^+(i) \cap T| +1)}{2}}\, .
\end{equation}
Second, the number of distinct objects obtained by sampling with replacement is always smaller than when taking the same number of samples without replacement. We use this to couple each set $N_{H_m}^+(i) \cap T$ with a set (of smaller or equal size) obtained by sampling with replacement $|N_{H_m}^+(i) \cap T|$ times from $T$. It follows that conditional on $\sum_{i \in E} |N_{H_m}^+(i) \cap T|$, the size $|N_{H_m}^+(E)\cap T|$ stochastically dominates $|S|$, where $S$ is a set of $\sum_{i \in E} |N_{H_m}^+(i) \cap T|$ independent, uniformly random elements of $T$. On the other hand, for such $S$ and for fixed $b < |T\cap [m]|$, if $\sum_{i \in E} |N_{H_m}^+(i) \cap T| = b$ then $|S|$ stochastically dominates a Binomial$(b,1-b/|T\cap [m]|)$ random variable. 
It thus follows from standard Binomial tail estimates (see Proposition~A.1) and the fact that 
$|T \cap [m]| \ge \hat{n}$, that if $\hat{n}\ge 4b$ then 
\begin{align*}
	& \quad \p{|N_{H_m}^+(E)\cap T| \le (b+1)/2 ~\left|~ \sum_{i \in E} |N_{H_m}^+(i) \cap T| = b\right.} \\
\le 	& \quad \p{\,\mathrm{Binomial}(b,b/\hat{n}) \ge  (b-1)/2\,} \\
\le 	& \quad \exp\left(-\frac{b-1}{2}\log\frac{\hat{n}}{4eb}\right)\, .
\end{align*}
We note that this upper bound is decreasing in $b$ for $b<\hat{n}/(4e^2)$, as can be straightforwardly checked. With (\ref{pigeonhole}), this yields 
\begin{align}\label{eq:keybound} 
& \quad \p{\sum_{i \in E} |N_{H_m}^+(i) \cap T|=b,E~\mbox{is not blocked}}\nonumber
\\
\le & \quad \p{\sum_{i \in E} |N_{H_m}^+(i) \cap T|=b}\cdot \exp\left(-\frac{b-1}{2}\log\frac{\hat{n}}{4eb}\right)\, ,
\end{align}
from which Proposition~\ref{lem:key} will follow essentially by union bounds and Binomial tail estimates. Some such estimates are encoded in the following straightforward bound, whose proof we defer to Appendix~\ref{Sec:easy_proofs}. 
\begin{lem}\label{lem:range of b}
Let $G$ be the event that for all $m \in [n',n]$ and all $E \subset [m]\setminus (Z^{\textsc{row}}(H_m)\cup I^+)$ with $|E| \le 1/(p \ln^{1/2} n)$, 
it is the case that $|E|\le \sum_{i \in E} |N_{H_m}^+(i) \cap T| < \hat{n}/(4e^2)$. Then 
\[
\p{G} 
= 1-O(n^{-\gamma})\, .
\]
\end{lem}
Now fix $E \subset [m] \setminus I^+$, and write $s=|E|$. Then $\sum_{i \in E} |N_{H_m}^+(i) \cap T|$ stochastically dominates a Binomial$(\hat{n}s,p_{\mathrm{min}})$ random variable (writing $p_{\mathrm{min}}=c\log n/n)$. 
It follows from the binomial tail bounds stated in Proposition~A.1 that 
\begin{equation}\label{eq:tail bound}
\p{\sum_{i \in E} |N_{H_m}^+(i) \cap T| \le 20 s} \le 
\left(\frac{e\hat{n}p_{\mathrm{min}}}{20 e^{\hat{n}p_{\mathrm{min}}/20}}\right)^{20s}
\le \left(\frac{\alpha c\log n}{n^{\alpha c/20}}\right)^{20s}\, ,
\end{equation}
the last inequality holding for $n$ large since $\hat{n} =n'-K(K+1) =\alpha n-K(K+1)$ 
and so $e^{\hat np_{\min}} \ge (e/20) e^{n' p_{\min}} = (e/20) n^{\alpha c}$. 
Next, observe that 
\begin{align}\label{eq:prob small}
&\quad \p{E~\mbox{is not blocked},G} \\ 
\le & \quad	 \p{s \le \sum_{i \in E} |N_{H_m}^+(i) \cap T|\le 20s,E~\mbox{is not blocked}} \nonumber \\
+&\quad  \p{20s \le \sum_{i \in E} |N_{H_m}^+(i) \cap T|< \hat{n}/(4e^2),E~\mbox{is not blocked}} \nonumber \, .
\end{align}

Taking a union bound over sets $E \subset [m]\setminus I^+$ with $|E|=s$ and using 
the bounds \eqref{eq:keybound} and \eqref{eq:tail bound} in (\ref{eq:prob small}), it follows that 
\begin{align}
& \p{D_{m,s},G} \nonumber\\
\le &\quad \underbrace{{m \choose s}  
\exp\left(-\frac{s-1}{2}\log\frac{\hat{n}}{4es}\right)\cdot \left(\frac{e\alpha c\log n}{n^{\alpha c/20}}\right)^{20s}}_{\text{(A)}}  
+ \underbrace{{m \choose s}  \exp\left(-4s\log\frac{\hat{n}}{80es}\right) \label{eq:bound large}}_{\text{(B)}}\, . \nonumber
\nonumber
\end{align}

Using that $\hat{n} \ge n/2$ for $n$ large and that ${m \choose s} \le (en/s)^s$,  we have 
\[
\text{(A)} \le \left(\frac{n}{4es}\right)^{1/2} \left( \frac{ (2e)^{3/2}  (e\alpha c \log n)^{20}}{n^{\alpha c-1/2} s^{1/2}} \right)^s
 \le \left( \frac{ (2e)^{3/2}  (e\alpha c \log n)^{20}}{16 n^{7\gamma/8} s^{1/2}} \right)^s\, .
\]
For $s \ge 4/\gamma =4/(\alpha c-1/2)$, we then have $n^{(\alpha c-1/2)s-1/2} = n^{\gamma s-1/2} \ge n^{7\gamma/8}$, so for such $s$ and for $n$ large, 
\[
\text{(A)} \le \left( \frac{ (2e)^{3/2}  (e\alpha c \log n)^{20}}{n^{7\gamma/8}} \right)^s\, .
\]
Again using that $\hat{n} \ge n/2$ for $n$ large and that ${m \choose s} \le (en/s)^s$, we have
\[
\text{(B)} \le \left(\frac{80^4 e^5 s^3}{n^3}\right)^s\, .
\]
The preceding bounds on (A) and on (B) are decreasing in $s$ for $4/\gamma \le s \le (p\ln^{1/2} n)^{-1}$, as can be verified by differentiation; it follows that 
\begin{align}
	& \p{G\cap \bigcup_{n' \le m \le n} \bigcup_{4/\gamma < s < (p\ln^{1/2} n)^{-1}}D_{m,s}} \nonumber\\
\le	& \sum_{4/\gamma < s < (p\ln^{1/2} n)^{-1}} \sum_{n' \le m \le n} [\text{(A)}+\text{(B)}] \nonumber\\
\le 	& (n-n'+1)(p\ln^{1/2} n)^{-1}  \left[\left( \frac{ 4  (e\alpha c \log n)^{20}}{n^{7\gamma/8}} \right)^{4/\gamma}
+ \left(\frac{80^4 e^5 (4/\gamma)^3}{n^3}\right)^{4/\gamma}\right] \nonumber\\
\le & n^2 \left[ \frac{ 4^{4/\gamma}  (e\alpha c \log n)^{80/\gamma}}{n^{7/2}} + \frac{O(1)}{n^{12/\gamma}} \right] \nonumber\\
= & O(n^{-1})\, ,\label{eq:smed}
\end{align}
since $\gamma =\alpha c-1/2 \in (0,1/4)$. 

We now treat the range $2 \le s \le 4/\gamma$; for this we require a final lemma. 
We say $H_n$ is {\em well-separated} if for all $n' \le m \le n$, 
for any distinct $u,v \in [m]\setminus I^+$, if $|N^+_{H_m}(u)| \le \ln \ln n$ and $|N^+_{H_m}(v)| \le \ln \ln n$ then there is no 
$2$-edge path (with edges of any orientation) joining $u$ and $v$ in $H_{m}$. 
\begin{lem}\label{lem:well-separated}
$\p{H_n\mbox{ is well-separated}} = 1-O(n^{-\gamma})$. 
\end{lem}
We defer the proof of Lemma~\ref{lem:well-separated} to Appendix~\ref{Sec:easy_proofs}, as it is essentially a reprise of an argument found in \cite{costello08symmetric} (though the results of \cite{costello08symmetric} do not themselves directly apply). 

Fix $n' \le m \le n$, $2 \le s \le 4/\gamma$ and $E \subset [m]\setminus I^+$ with $|E|=s$. 
Write $\hat{G}= G \cap \{H_m\mbox{ is well separated for all }n' \le m \le n\}$. 
Arguing as at (\ref{eq:prob small}), we have 
\begin{align*}\label{eq:prob small}
&\quad \p{E~\mbox{is not blocked},\hat{G}} \\ 
\le & \quad	 \p{s \le \sum_{i \in E} |N_{H_m}^+(i) \cap T|\le 20s,E~\mbox{is not blocked},\hat{G}} \nonumber \\
+&\quad  \p{20s \le \sum_{i \in E} |N_{H_m}^+(i) \cap T|< \hat{n}/(4e^2),E~\mbox{is not blocked}} \nonumber \, .
\end{align*}
Now note that if $\sum_{i \in E} |N_{H_m}^+(i) \cap T|\le 20s \le 80/\gamma$ then all vertices in $E$ have degree at most
$80/\gamma + |[m]\setminus T| \le 80/\gamma + K(K+1)$. For $n$ large enough that $80/\gamma+K(K+1) \le \ln \ln n$, if $H_m$ is well-separated then the sets $\{|N_{H_m}^+(i) \cap T|: i \in E\}$ are disjoint. Since $s \ge 2$, it follows that in this case $E$ {\em is} blocked so the first probability on the right hand side of the preceding bound is zero.  By a union bound 
and the same argument used to bound $\text{(B)}$, above, it follows that 
\begin{align*}
& \p{\hat{G} \cap \bigcup_{n' \le m \le n} \bigcup_{2 \le s \le 4/\gamma}D_{m,s}} \\
\le & (n+1) \cdot \frac{4}{\gamma}  \cdot \left(\frac{80^4 e^5 2^3}{n^3}\right)^2 \\
= & O(n^{-5})\, .
\end{align*}
Combining this bound with (\ref{eq:slarge}), (\ref{eq:smed}) and Lemmas~\ref{lem:range of b} and~\ref{lem:well-separated} then yields 
\begin{align*}
& \p{\bigcup_{n' \le m \le n} \bigcup_{2 \le s \le \ln\ln n/(2p)} D_{m,s}} \\
\le & \exp\left(3\ln n - n/(2\ln n)\right) + O(n^{-1}) + O(n^{-5}) + O(n^{-\gamma}) + O(n^{-\gamma})\, \\
= & O(n^{-\gamma})\, ,
\end{align*}
which, recalling (\ref{eq:key_toprove}), completes the proof of Proposition~\ref{lem:key}. 
\qed

\section*{Acknowledgements}
LAB received support from NSERC and FQRNT, and LE was supported by CONACYT scholarship 309052, for the duration of this research. We thank all three agencies for their support. We also thank two anonymous referees for their careful reading of the paper.

\appendix

\section{Binomial tail bounds} 
In this section we recall standard binomial tail bounds. The bounds in the following proposition are contained in \cite{penrose}, Lemma 1.1 and \cite{mcdiarmid}, Theorem 1.1.
\begin{prop}\label{chernoff}
For $m \in \N$ and $0 \le q \le 1$, if $X\eqdist \mathrm{Binomial}(m,q)$ then writing $\mu=mq$, for $k \ge \mu$ we have 
\[
\p{X \ge k} \le \exp(-\mu- k\ln(k/(e\mu)))\, ,
\]
for $k \le \mu$ we have 
\[
\p{X \le k} \le \exp(-\mu- k\ln(k/(e\mu)))\, ,
\]
and for $\eps > 0$ we have 
\[
\p{X-\mu > \eps m} \le \exp(-2 \eps^2m)\, , \quad \p{X-\mu <- \eps m} \exp(-2 \eps^2m)\, .
\]
\end{prop}

\section{Remaining proofs}\label{Sec:easy_proofs}

\begin{proof}[Proof of \refL{lem:rest}]
Fix $\eps > 0$ and let $a>1$ large enough that $1-e^{-e^{-a}}<\eps/4$ and that $6(e/4)^{e^a}<\eps /2$. Recall that $p_1=\frac{\ln n-a}{n}$ and $p_2=\frac{\ln n +a}{n}$. 

Observe that $\tau(\cH_n)> p_2$ if either $Z^{\textsc{row}}(H_{n,p_2})$ or $Z^{\textsc{col}}(H_{n,p_2})$ is non-empty. As both $|Z^{\textsc{row}}(H_{n,p_2})|$ and 
$|Z^{\textsc{col}}(H_{n,p_2})|$ are asymptotically Poisson($e^{-a}$), our choice of $a$ yields that
\[\p{\tau(\cH_n) > p_2}\le 2\left(1-e^{-e^{-a}}\right)+o(1)\le \frac{\eps}{2}+o(1).\]
This gives the first bound of the lemma. For the second bound, we claim that it suffices to prove $\p{\cD_K(p_1,p_2)}\ge 1- \eps/2$. 

Indeed, assuming this bound, since $\cD_K(p_1,p_2)\cap \{\tau(\cH_n)=p_2\}\subset \cD_K(p_1,\tau(\cH_n))$, we have 
\[\p{\cD_K(p_1,\tau(\cH_n))}
\ge \p{\cD_K(p_1,p_2)} -\p{\tau(\cH_n)>p_2}
\ge 1-\eps+o(1).
\]
Given $i,j\in [n]$, if $U_{ij}>p_1$, then $e=ij\notin H_{n,p_1}$. We have 
\[\p{e\in H_{n,p_2}~|~e\notin H_{n,p_1}} 
= \p{U_{ij}\leq p_2~|~U_{ij}> p_1}\leq \frac{p_2-p_1}{1-p_1}\le \frac{4a}{n}.\]
We use this estimate to study $(N_{H_{n,p_2}}^+(i))_{i\in Z^{\textsc{row}}(H_{n,p_1})}$. For $i,j \in [n]$, by the preceding bound and a union bound,
\[\p{N_{H_{n,p_2}}^+(i)\cap N_{H_{n,p_2}}^+(j)\neq \emptyset~|~i,j \in Z^{\textsc{row}}(H_{n,p_1})}\le \frac{n(4a)^2}{n^2}=\frac{(4a)^2}{n}.\]
By another union bound, for any fixed set $Z^+\in [n]$ it follows that 
\begin{equation}\label{eq:disjoint}
\p{\{N_{H_{n,p_2}}^+(i)\}_{i\in Z^+}
\text{ are not pairwise disjoint}~|~Z^{\textsc{row}}(H_{n,p_1})=Z^+}
\le \frac{8a^2|Z^+|^2}{n}.
\end{equation}
Similarly, for any fixed $Z^+,Z^-\in [n]$
\begin{equation}\label{eq:stable}
\p{\bigcup_{i\in Z^+} N_{H_{n,p_2}}^+(i)
\cap Z^- \neq \emptyset~|~Z^{\textsc{row}}(H_{n,p_1})=Z^+,Z^{\textsc{col}}(H_{n,p_1})=Z^-}
\le \frac{4a|Z^+||Z^-|}{n}.
\end{equation}
Finally, given that $i\in Z^{\textsc{row}}(H_{n,p_1})$ we have $|N_{H_{n,p_2}}^+(i)| \preceq_{\mathrm{st}} \mathrm{Bin}(n,4a/n)$. It follows by a Chernoff bound that 
\[
\p{|N_{H_{n,p_2}}^+(i)| > 8a~|~i \in Z^{\textsc{row}}(H_{n,p_1})} \le e^{-16a^2}\, , 
\]
By a union bound, for any $Z+ \subset [n]$, 
\begin{equation}\label{eq:size}
\p{\bigcup_{i\in Z^+}|N_{H_{n,p_2}}^+(v)| > 8a~|~Z^{\textsc{row}}(H_{n,p_1})=Z^+} \le |Z^+|e^{-16a^2}.
\end{equation}
Observe that, a similar argument conditioning on $Z^{\textsc{col}}(H_{n,p_1})=Z^-$ gives the same bounds in \eqref{eq:disjoint}~and\eqref{eq:size} for the sequence $(N_{H_{n,p_2}}^-(i))_{i\in Z^{\textsc{col}}(H_{n,p_1})}$. Additionally, 
$\bigcup_{i\in Z^{\textsc{row}}(H_{n,p_1})} N_{H_{n,p_2}}^+(i)
\cap Z^- \neq \emptyset$ implies $\bigcup_{i\in Z^{\textsc{col}}(H_{n,p_1})} N_{H_{n,p_2}}^-(i) \cap Z^{\textsc{row}} \neq \emptyset$. 

So far the bounds obtained depend on the size of fixed sets $Z^+, Z^-\in [n]$. Now let $\cK$ be the event that both 
$Z^{\textsc{row}}(H_{n,p_1})$ and $Z^{\textsc{col}}(H_{n,p_1})$
have size at most $K=K(a)=\lfloor 2e^a\rfloor$. We claim that
\begin{equation}\label{eq:sizeZ}
\p{\cK}=\p{|Z^{\textsc{row}}(H_{n,p_1})|,|Z^{\textsc{col}}(H_{n,p_1})|\le K}\ge 1-2(e/4)^{e^a}+o(1).
\end{equation}
For this, we use that if $X\stackrel{d}{=} Poisson(\lambda)$, then $\p{X \ge 2\lambda} \le (e/4)^{\lambda}$; see, e.g., Lemma~1.2 of \cite{penrose}. 
Since $|Z^{\textsc{row}}(H_{n,p_1})|,|Z^{\textsc{col}}(H_{n,p_1})|$ are asymptotically Poisson($\lambda$),
\eqref{eq:sizeZ} follows by a union bound.

We now bound $\cD_K(p_1,p_2)$ using the above inequalities. If $\cK$ occurs, then there exist (possibly empty) sets $Z^+,Z^-\in [n]$ of size at most $K$. By \eqref{eq:disjoint},\eqref{eq:stable},\eqref{eq:size} we then obtain
\begin{equation}\label{eq:sizeZ2}
\p{\cD_K(p_1,p_2) \cup \{\tau(\cH_n)\le p_1\} ~|~\cK}\geq 1-\frac{16a^2K^2+4aK^2}{n}-2Ke^{-16a^2}.
\end{equation}
Note that if $\tau(\cH_n) \le p_1$ then $|Z^{\textsc{row}}(H_{n,p_1})|+|Z^{\textsc{col}}(H_{n,p_1})|=0$. Thus, 
\begin{align*}
&\p{\cD_K(p_1,p_2)} \\
\ge & \p{\cD_K(p_1,p_2) \cup \{ \tau(\cH_n)\le p_1\}~|~\cK}\p{\cK} -\p{\tau(\cH_n)\le p_1}\\
\ge & (1-2Ke^{-16a^2}-o(1))(1-2(e/4)^{e^a})-2e^{-e^a}+o(1)\\
\ge & 1-6(e/4)^{e^a}\, .
\end{align*}
In the second inequality we use \eqref{eq:sizeZ},\eqref{eq:sizeZ2} and the fact that $|Z^{\textsc{row}}(H_{n,p})|$ and $|Z^{\textsc{row}}(H_{n,p})|$ are asymptotically Poisson($e^a$); 
the final inequality then follows from the fact that $a>1$ and $K=\lfloor 2e^a\rfloor$ by straightforward calculation. By our choice of $a$, the final bound is at least $1-\eps/2+o(1)$, completing the proof.
\end{proof}

\begin{proof}[Proof of Lemma~\ref{lem:range of b}]
We first bound the maximum degree of vertices in $[n]\setminus I^+$. A union bound together with a Chernoff bound yields 
\[\p{\exists v\in [n]\setminus I^+:\, |N^+_{H_n}(v)|> 2np}\le ne^{-(np)^2}\le n^{1-c^2\ln n},\]
where the last inequality uses that $p \ge p_{\min}=c\ln n/n$. It follows that with high probability, for any $n'\le m\le n$,
\[\sum_{i\in E} |N^+_{H_m}(i)\cap T|\le 2|E|np\le 2n/\ln^{1/2}n \le \hat{n}/4e^2.\]

To obtain the lower bound note that 
\[\sum_{i\in E} |N^+_{H_m}(i)\cap T|\ge \sum_{i\in E} |N^+_{H_{m}}(i)\cap T\cap [n']|.\]
Thus, it suffices to show that 
\[\p{\forall v\in [n]\setminus I^+:\, |N^+_{H_{n}}(v)\cap T\cap [n']|\neq \emptyset}\ge 1-O(n^{-\gamma}).\]
For fixed $v \in [n]\setminus I^+$, since $|[n']\cap T|\ge n'-K(K+1)$, we have 
\begin{align*}
 \p{|N_{H_{n}}^+(v)\cap T\cap [n']|=0} \le (1-p)^{n'-(K+1)^2} \le Ce^{-\alpha pn}=O(n^{-\gamma-1/2}).
\end{align*}
The bound above applies in particular for vertices in $\cup_{i\in I^-}S^-_i$, and there are at most $K^2=O(1)$ such vertices. 
On the other hand, if $v\in [n]\setminus(\cup_{i\in I^-} S_i^- \cup I^+)$ and $N^+_{H_{n}}(v)\cap [n']\neq \emptyset$, then either $|N^+_{H_{n}}(v)\cap T\cap [n']|\neq \emptyset$ or $|N^+_{H_n}(v)\cap U\cap [n']|\neq \emptyset$. It thus remains to bound 
\[\p{\exists v \in [n] \setminus I^+: |N_{H_{n}}^+(v)\cap U| \ge 1, | N_{H_{n}}^+(v)\cap T\cap [n']|=0}, \]
which is $O(n^{-\gamma})$ by \eqref{eq:use later}. 
\end{proof}

\begin{proof}[Proof of Lemma~\ref{lem:well-separated}]
We say a vertex $v$ has {\em low degree} in $H_m$ if $N^+_{H_m}(v)\le d := \ln\ln n$. 

First consider the graph $H_{n'}$. The event that fixed vertices $v_1$ and $v_2$ are connected by a 2-path is monotone increasing,  while the event that both vertices have low out-degree is monotone decreasing. By the FKG inequality, these events are negatively correlated and so the probability that both events hold is bounded from above by the product of their probabilities.

The random variable $|N_{H_{n'}}^+(v_1)|$ is Binomial$(n'-1,p)$ distributed; we will bound $\p{|N_{H_{n'}}^+(v_1)| \le d}$. We have 
\begin{align*}
\p{|N_{H_{n'}}^+(v_1)| \le d}\quad =
& \quad  \sum_{i=0}^{d} {n'-1 \choose i} (p)^i(1-p)^{n'-i-1} \\ 
\le&\quad (1-p)^{n'}\sum_{i=0}^{d} 2(2n'p)^i \\
\le&\quad e^{-n'p}(2n'p)^{d+1} 
\end{align*}
where in the first inequality we use that $1-p \geq \frac{1}{2}$. Also, the random variables $|N_{H_{n'}}^+(v_1)|$ and $|N_{H_{n'}}^+(v_2)|$ are iid and hence $\p{|N_{H_{n'}}^+(v_1)| \le d,|N_{H_{n'}}^+(v_2)|\leq d} \le e^{-2n'p}(2n'p)^{2(d+1)}$. 

Since $\p{|N_{H_{n'}}^+(v_1)| \le d}$ is decreasing in $p$, the preceding bound implies that if $p \ge (2/\alpha)\ln n/n$ then 
$\p{|N_{H_{n'}}^+(v_1)| \le d,|N_{H_{n'}}^+(v_2)|\leq d} \le n^{-4} (4 \ln n)^{2(d+1)}$. Also, the same bound holds for $\p{|N_{H_{m}}^+(v_1)| \le d,|N_{H_{m}}^+(v_2)|\leq d}$ for any $m \in [n',n]$, since $|N_{H_{m}}^+(v_1)|$ and $|N_{H_{m}}^+(v_2)|$ are increasing in $m$. In this case the bound claimed in the statement of the lemma holds by a union bound over $m \in [n',n]$ and over pairs of vertices $v_1,v_2$ of $H_m$. 

We thus assume for the remainder of the proof that $p \le (4/\alpha) \ln n/n$. We still also have $p \ge p_{\mathrm{min}}=c\ln n/n$, so 
the preceding argument yields that for any $m \in [n',n]$, 
\begin{align*}
\p{|N_{H_{m}}^+(v_1)| \le d,|N_{H_{m}}^+(v_2)|\leq d}\le  \frac{(2\ln n)^{2d+2}}{n^{2c\alpha}}\leq \frac{\ln^{3d} n}{n^{1+2\gamma}};
\end{align*}

On the other hand, the probability that $v_1$ and $v_2$ are adjacent or have a common neighbour is at most $2p+4np^2$. By a union bound and the upper bound on $p$, the probability that $H_{n'}$ contains two low degree vertices with a common neighbour is therefore at most 
\begin{align*}
{n' \choose 2} \frac{\ln^{3d} n(2p+4np^2)}{n^{1+2\gamma}}  \leq \frac{72\ln^{3d+2} n}{n^{2\gamma}}.
\end{align*} 

We next consider $H_m$ with $m>n'$.
Let $z=m$ be the unique vertex of $H_m$ not in $H_{m-1}$. If $H_m$ is the first graph which is not well separated, then it either (a) $z$ is adjacent to two low degree vertices $v_1$, $v_2$ which are neither connected by a two-edge path nor adjacent in $H_{m-1}$ or (b) $z$ is itself a low degree vertex at (undirected) distance 1 or 2 of a second vertex $v_0$ of low out-degree. By a union bound over the pairs of vertices in $H_{m-1}$ we obtain that (a) occurs with probability at most 
\begin{align*}
{m-1 \choose 2} \frac{4p^2\ln^{3d} n}{n^{1+2\gamma}}\leq \frac{(64/\alpha^2)\ln^{3d+2} n}{n^{1+2\gamma}}
\end{align*}
Similarly,
since $z$ and any fixed vertex $v_0\in [m-1]$ are at distance 1 or 2 with probability less than $2p+4np^2$, a union bound over the vertices of $H_{m-1}$ implies that (b) occurs with probability at most
\begin{align*}
\frac{(m-1)\ln^{3d} n (2p+4np^2)}{n^{1+2\gamma}}\leq \frac{(72/\alpha^2)\ln^{3d+2} n}{n^{1+2\gamma}},
\end{align*}
Combining these bounds and summing over $m \in [n',n]$, we obtain that
\[
\p{H_n \mbox{ is well separated}}=O(\ln^{3d+2} n/n^{2\gamma})\, ,
\]
and the latter term is $O(n^{-\gamma})$, as required. 
\end{proof}

\end{document}